\documentclass[10pt]{amsart}
   \usepackage{amsthm,amsmath,amssymb,amscd,graphicx,enumerate, stmaryrd,xspace,verbatim, epic, eepic,color,url}
\usepackage{eurosym}
\usepackage[backref=page]{hyperref}
\usepackage{pgf}
\usepackage{amsmath}
\usepackage{amsthm}
\usepackage{amsfonts}
\usepackage{mathrsfs}
\usepackage{amssymb}
\usepackage{amscd}
\usepackage[all]{xy}
\usepackage{tikz}
\usepackage{amscd}
\usepackage{graphicx}
\usepackage{setspace}
\usepackage{enumerate}
\usepackage{xcolor}
\usepackage[
  hscale=0.7,
    vscale=0.75,
    headheight=13pt,
    centering,
    ]{geometry}

\makeindex

\usepackage{color, colortbl}

    \newtheorem{Lem}{Lemma}[section]
    \newtheorem{Lem-Def}[Lem]{Lemma-Definition}
    \newtheorem{Prop}[Lem]{Proposition}
       
     \newtheorem*{thm*}{Theorem}

    \newtheorem{Thm}[Lem]{Theorem}

\theoremstyle{definition}

    \newtheorem{Def}[Lem]{Definition}

    \newtheorem{Rem}[Lem]{Remark}

\newcommand{\mc}{\mathcal}

\newcommand{\G}{\Gamma}
\newcommand{\tit}{\textit}

\newcommand{\ora}[1]{\overrightarrow{#1}}

\newcommand{\N}{\mathcal N}
\newcommand{\T}{\mathcal T}

\renewcommand{\L}{\mathcal L}
\renewcommand{\O}{\mathcal O}
\newcommand{\C}{\mathcal C}

\newcommand{\D}{\mathcal D}

\newcommand{\R}{\mathbb R}
\renewcommand{\P}{\mathcal{P}}
\newcommand{\col}{\colon}

\newcommand{\Ps}{\mathbb{P}}
\newcommand{\ra}{\rightarrow}
\newcommand{\ol}{\overline}
\newcommand{\supp}{\text{supp}}

\newcommand{\wt}{\widetilde}

\newcommand{\wh}{\widehat}
\newcommand{\dra}{\dashrightarrow}

\newcommand{\J}{\mathcal{J}}

\newcommand{\lra}{\longrightarrow}

\DeclareMathOperator{\Sym}{Sym}

\DeclareMathOperator{\st}{st}

\DeclareMathOperator{\qs}{qs}

\DeclareMathOperator{\Spec}{Spec}

\DeclareMathOperator{\free}{free}

\DeclareMathOperator{\Proj}{Proj}

\newcommand{\trop}{\text{trop}}
\renewcommand{\div}{\textnormal{div}}

\newcommand{\boundellipse}[3]% center, xdim, ydim
{(#1) ellipse (#2 and #3)
}

\title[Degree-2 Abel maps and hyperelleptic curves]{Degree-2 Abel maps and hyperelleptic curves}

\author{Alex Abreu, Sally Andria, and Marco Pacini}

\begin{document}

\maketitle

\begin{abstract}
 In this paper we resolve the degree-2 Abel map for nodal curves. Our results are based on a previous work of the authors reducing the problem of the resolution of the Abel map to a combinatorial problem via tropical geometry. As an application, we characterize when the (symmetrized) degree-2 Abel map is not injective, a property that, for a smooth curve, is equivalent to the curve being hyperelliptic. 
\end{abstract}

\bigskip



%\tableofcontents

\section{Introduction}

%-Consequencias para as equivalencias de cima (citar a weakly hyperell.).

This paper is dedicated to the construction of an explicit resolution of the degree-$2$ Abel-Jacobi map for a regular smoothing of a nodal curve. For a smooth curve $C$, the degree-$d$ Abel map is an important morphism taking a $d$-tuple of points on $C$, to the associated invertible sheaf on the curve (tensored with a fixed invertible sheaf on the curve). When the fixed invertible sheaf is $\mathcal O_C(dP_0)$ for a pont $P_0$ on $C$, the map is usually called Abel-Jacobi map. This map encodes many important properties of the curve. For instance,  the degree-$2$ Abel map detects when the curve is hyperelliptic. More precisely, a smooth curve is hyperelliptic if and only the degree-$2$ Abel map is not injective, and in this case the curve is endowed with a $g^1_2$, which can be identified with the fiber of the Abel map (up to the action of the symmetric group). In this paper we investigate how we can extend the above description to singular curves. 

We construct an explicit resolution of the degree-$2$ Abel-Jacobi map using the results in \cite{AAP}, where the general problem of resolving Abel maps is reduced to checking a certain combinatorial property of the tropical Abel map. More precisely, this  translates into the problem of showing the existence of a compatibility between the polyhedral structures of the tropical Jacobian of a curve and the product of the relevant tropical curve. This is the result contained  in \cite[Theorem A]{AAP}. In degree $2$, this yields a very explicit combinatorial condition describing how to blow up the domain of the geometric Abel map to get a resolution. This is summarized in Theorem \ref{thm:tropgeoAbel}.\par

 Let $\pi\col \C\to B$ be a regular smoothing of a curve $C$ with  a section  $\sigma$  through its smooth locus, and $\mu$ be a polarization on $\C$. We denote by $\overline{\J}_{\mu}^{\sigma}$ the Esteves compactified Jacobian parametrizing $(\sigma,\mu)$-quasistable torsion-free rank-$1$ sheaves on $\C$ (see \cite{EE01}). As usual, we write $\C^2:= \C\times_B\C$.
Let $\L$ be an invertible sheaf on $\C/B$ of degree-$(k+2)$. We define the degree-$2$ Abel (rational) map $\alpha^2_{\L}$ as
\begin{align*}
\alpha^2_{\L}\col \C^2 &\dashrightarrow \ol{\J}_{\mu}^{\sigma}\\
(Q_1,Q_2)&\longmapsto [\L|_{\pi^{-1}(Q_1)}(-Q_1-Q_2)].
\end{align*}

Our main result holds when $\mu$ is the trivial degree-$0$ polarization and $\L$ is the trivial sheaf $\O_{\mathcal C}$. An important ingredient to describe the resolution of the degree-$2$ Abel-Jacobi map is the notion of tail of a nodal curve.
A subcurve of a nodal curve is a $\delta$-tail if the subcurve and its complementary curves are connected and intersect each other in $\delta$ nodes. 
  
  \begin{thm*}[Theorem \ref{thm:Abel2}]
  %Let $\pi\col\C\ra B$ be a regular smoothing of a nodal curve $C$ with smooth components.
  %Let $\sigma\col B\ra \C$ be a section of $\pi$ through its smooth locus and
  Let $Z_1,\dots,Z_N$ be the 2-tails and the 3-tails of $C$ which do not contain $\sigma(0)$.  Consider the sequence of blowups
  \[
\wt\C^2_N\stackrel{\phi_N}{\lra}\cdots \stackrel{\phi_2}{\lra}\wt\C^2_1\stackrel{\phi_1}{\lra}\wt\C^2_0\stackrel{\phi_0}{\lra}\C^2
  \]
  where $\phi_0$ is the blowup of $\C^2$ along its diagonal subscheme and $\phi_i$ is the blowup of $\wt{\C}^2_{i-1}$ along the strict transform of the divisor $Z_i\times Z_i$ of $\C^2$ via $\phi_0\circ\cdots\circ \phi_{i-1}$. Then the rational map
  \[
\alpha^2_{\mc O_\C}\circ\phi_0\circ\cdots\circ\phi_N\col \wh\C^2_N\dra\overline{\mc J}^\sigma_\mu
  \]
  is a morphism, i.e., it is defined everywhere.
  \end{thm*}

Next we investigate the relation between the degree-2 Abel map and hyperelliptic (nodal) curves. More precisely, we study when the (symmetrized) degree-2 Abel map is not injective. The upshot is that this happens exactly when the curve has a simple torsion-free rank-1 sheaf of degree 2 with non-negative degree over every component of the curve and at least two sections. We call a curve satisfying all these condition a \emph{pseudo-hyperelliptic} curve. It is easy to see that if a curve is hyperelliptic, then it is pseudo-hyperelliptic. 

It is worth noticing that a variation of the condition of hyperelliptic curve was already given by Caporaso in \cite{Capohyper}. She introduced and study the notion of weakly-hyperelliptic curve, which is the condition of the existence of a balanced degree-2 invertible sheaf on a curve with at least 2 sections. Again, if a curve is hyperelliptic, then it is weakly-hyperelliptic. We study the relation between weakly-hyperelliptic and pseudo-hyperelliptic.

\begin{thm*}[Theorem \ref{thm:hyperelliptic}]
Let $C$ be a curve with no separating nodes. The following properties hold.
\begin{enumerate}
    \item The curve $C$ is pseudo-hyperelliptic
if and only if, for some (every) regular smoothing $\C\to B$ of $C$, the symmetrized degree-$2$ Abel map of $\mathcal C$ is not injective.
\item If $C$ is stable and weakly-hyperelliptic, then $C$ is pseudo-hyperelliptic.
\end{enumerate}
\end{thm*}

%TERMINAR A INTRO FALANDO DE CURVAS HIPERELIPTICAS (CAPORASO)

\section{Preliminaries}

Throughout the paper, we will use the notations introduced in \cite[Sections 2 and 3]{AAMPJac} and \cite[Section 3]{AAP}. In this section we just recall some basic definitions and constructions.

%\textcolor{red}{$X$ tropical curve $X_\Gamma$} 

Given a graph $\G$, we denote by $V(\G)$  and $E(\G)$ the sets of vertices and edges of $\G$. Given a subset $V\subset V(\Gamma)$, we set $V^c=V(\Gamma)\setminus V$. For an orientation $\overrightarrow \G$ on $\G$, we denote by $s(e)$ and $t(e)$ the source and target of an edge $e\in E(\G)$. Given subsets $V,W\subset V(\G)$, we let $E(V,W)$ be the set of edges of $\G$ connecting a vertex of $V$ with a vertex of $W$.
A \emph{refinement} of a graph $\G$ is a graph obtained by inserting a non-negative number $n_e$ (depending on $e$) of vertices in the interior of each edge $e$ of $\G$. If a vertex $v$ of the refinement is inserted in the interior of an edge $e$ of $\G$, we say that $v$ is a vertex \emph{over} $e$.

A \tit{metric graph} is a pair $(\G,\ell)$ where $\G$ is a graph and $\ell$ is a function $\ell\col E(\G)\ra\R_{>0}$, called the \tit{length function}.
Let $(\G,\ell)$ be a metric graph. If $\ora{\G}$ is an orientation on $\G,$ we define the \tit{tropical curve} $X$ associated to $(\ora{\G},\ell)$ as
\[X=\frac{\left(\bigcup_{e\in E(\ora{\G})}I_{e}\cup V(\ora{\G})\right)}{\sim}\]
where $I_{e}=[0,\ell(e)]\times\{e\}$ and $\sim$ is the equivalence relation generated by $(0,e)\sim s(e)$ and $(\ell(e),e)\sim t(e).$ 
We usually just write $e$ to represent $I_e$ in $X$, and denote by $e^\circ$  the interior of $e$. We say that $(\Gamma,\ell)$ is a \emph{model} of the tropical curve $X$. We will identify tropical curves with isometric models.

Let $\G$ be a graph and define $\ell\col E(\G)\to \mathbb{R}$ by $\ell(e)=1$ for every $e\in E(\G)$. We denote by $X_{\G}$ the tropical curve induced by the metric graph $(\G,\ell)$.

Let $X$ be a tropical curve and $\Gamma$ be a graph. 
%\textcolor{red}{$supp(\mathcal D)$}
A \tit{divisor} on $X$ (respectively, on $\Gamma$) is a function $\mc{D}\col X\ra\mathbb{Z}$ (respectively, $D\col V(\Gamma)\to \mathbb{Z}$) such that $\mc{D}(p)\neq0$ only for finitely many points $p\in X.$ Given a divisor $\mathcal D$ on $X$, we define the \tit{support} of $\mc{D}$ as the set of points $p$ of $X$ such that $\mc{D}(p)\neq0$ and denote it by $\supp(\mc{D})$. A \emph{polarization} on $X$ (respectively, on $\Gamma$) is a function $\mu\col X\ra \mathbb R$ (respectively, $\mu\col V(\Gamma)\to \mathbb{R}$) such that $\mu(p)\ne 0$ only for finitely many points $p\in X$ and such that $\sum_{p\in X} \mu(p)$ (respectively, $\sum_{v\in V(v\in V(\Gamma)} \mu(v)$) is an integer, called the \emph{degree} of the polarization $\mu$.\par
   Given a point $p_0$ in $X$, (respectively, a vertex $v_0\in V(\Gamma)$), a divisor $\mathcal{D}$ on $X$ (respectively, $D$ on $\Gamma$) is called $(p_0,\mu)$-quasistable (respectively, $(v_0,\mu)$-quasistable) if:
   \[
   \sum_{p\in Y}(\mathcal{D}(p)-\mu(p)) + \frac{\delta_Y}{2}\geq 0
   \]
   for every tropical subcurve $Y$ of $X$ (respectively, every subset $Y\subset V(\Gamma)$), where the inequality is strict if $p_0\in Y\neq  X$. Here, $\delta_Y$ is the number of tangent direction outgoing from $Y$ in the case of a tropical curve (see \cite[Section 3.1]{AAMPJac} for the precise definition), while it is equal to $|E(Y,Y^c)|$ in the case of a graph.

 Let $X$ be a tropical curve and $p_0$ be a point of $X$. Let $\mu$ be a polarization on $X$. 
 Recall that in an equivalence class of a divisor on a tropical curve there is just one $(p_0,\mu)$-quasistable divisor (see \cite[Theorem 5.6]{AAMPJac}). For a degree-$d$ divisor $\D$ on $X$, we denote by $\qs(\D)$ the unique $(p_0,\mu)$-quasistable divisor on $X$ which is equivalent to $\D$.
Given an oriented model $(\Gamma,\ell)$ of $X$, for every edge $e\in E(\G)$ and every real number $t\in [0,\ell(e)]$, we let by 
$p_{e,t}$ the point on $e$ at distance $t$ from the source of $e$.\par
  Given a tropical curve $X$, we let  $J^{\trop}_{p_0,\mu}(X)$ be the tropical Jacobian parametrizing $(p_0,\mu)$-quasistable divisors on $X$. Recall that $J^{\trop}_{p_0,\mu}(X)$ is homeomorphic to the usual tropical Jacobian (see \cite[Theorem 5.10]{AAMPJac}). 
   We set $X^2:=X\times X$.
Given a divisor $\D^\dagger$ on $X$, we define the tropical Abel map
\begin{align*}
 \alpha_{2,\D^\dagger}^\trop
\col X^2 &\to J^{\trop}_{p_0,\mu}(X)\\
(p_1,p_2)&\longmapsto [\D^\dagger-p_1-p_2],
\end{align*}
where $[-]$ denotes the class of a divisor in the tropical Jacobian. Alternatively, the map $\alpha_{2,\D^\dagger}^\trop$ takes $(p_1,p_2)$ to the unique $(p_0,\mu)$-quasistable divisor in the class of $\D^\dagger-p_1-p_2$.

\begin{Rem}
\label{prop:quasiquasi}
Let $X$ be a tropical curve with a point $p_0\in X$. Let $\Gamma$ be a model of $X$. Let $\mu$ be a polarization on $X$ induced by a polarization on $\Gamma$ and $\D$ a degree-$d$ divisor on $X$.
%such that $\supp(\D)\subset V(\Gamma)$. 
We let $\widehat{\Gamma}$ the minimal refinement of $\Gamma$ such that $\supp(\D)\subset V(\widehat{\Gamma})$. We denote by $D$ the divisor on $\widehat{\Gamma}$ induced by $\mathcal D$. We call the pair $(\widehat{\Gamma},D)$ on $\widehat{\Gamma}$ the \emph{combinatorial type} of $\mathcal D$.
By  \cite[Proposition 5.3]{AAMPJac}, the degree-$d$ divisor $\D$ on $X$ is $(p_0,\mu)$-quasistable if and only if $\widehat{\Gamma}$ is obtained by inserting at most one vertex in the interior of each edge of $\Gamma$ and $D$ is $(p_0,\mu)$-quasistable on $\widehat{\Gamma}$. 
\end{Rem}

\section{Degree-2 Abel maps}\label{sec:degree2}

Let $C$ be a nodal curve over an algebraically closed filed $k$. A subcurve $Z$ of $C$ is a reduced union of components of $C$. Given a subcurve $Z$ of $C$, we let $Z^c:=\overline{C\setminus Z}$. Throughout this section we will fix a regular smoothing $\pi\col\C\ra B$ of a nodal curve $C$ with a section $\sigma\col B\ra \C$ of $\pi$ through its smooth locus.  We denote by $\C^2:=\C\times_B\C$.

Let $\mu$ be a degree-$k$ polarization on $\C$. We denote by $\ol{\J}_{\mu}^{\sigma}$ the Esteves compactified Jacobian parametrizing $(\sigma,\mu)$-quasistable torsion-free rank-$1$ sheaves on the curves of the family $\pi$ (see \cite{EE01} for more details).
Let $\L$ be an invertible sheaf on $\C/B$ of degree-$(k+2)$. As in \cite{AAJCMP}, we define the degree-$2$ Abel (rational) map $\alpha^2_{\L}$ as
\begin{align*}
\alpha^2_{\L}\col \C^2 &\dashrightarrow \ol{\J}_{\mu}^{\sigma}\\
(Q_1,Q_2)&\longmapsto [\L|_{\pi^{-1}(\pi(Q_1))}(-Q_1-Q_2)].
\end{align*}

We let $\Gamma$ be the dual graph of $C$ and $X_{\Gamma}$ be the tropical curve induced by $\Gamma$ (with unitary lengths). Given an invertible sheaf $\L$ on $\C$, we denote by $D^{\dagger}_{\L}$ the divisor on $\G$ given by the multidegree of $\L|_C$. We also let $\D^\dagger_\L$ be the divisor on $X_\G$ induced by $D^{\dagger}_{\L}$.

 %If $\C\to B$ is a smoothing of a nodal curve $C$ with dual graph $\G$, then $X_{\G}$ is what is usually called the ``tropical curve associated to the smoothing". Notice, however, that many results in this section hold for more general length functions.

Given a point $\N=(N_1,N_2)$ of $\mathcal C^2$, where $N_i$ is a  node of $C$, we will consider the following two ways of blowing up $\C^2$ locally around $\N$. If $N_1\in Z_1\cap Z_1^c$ and $N_2\in Z_2\cap Z_2^c$ for subcurves $Z_1$ and $Z_2$ of $C$, we can consider the blowups $\phi\col \wt{\mc \C}^2_\phi\ra \C^2$ and $\phi'\col \wt{\mc \C}^2_{\phi'}\ra \C^2$  respectively along $Z_1\times Z_2$, or along $Z_1\times Z_2^c$. The first one is also equivalent to the blowup along $Z_1^c\times Z_2^c$ and the second one is equivalent to the blowup along $Z_1^c\times Z_2$. In both cases, the inverse image of $\mc N$ is isomorphic to $\mathbb P^1_k$. The situation is illustrated in Figure \ref{Fig:blowup}, where $\st_\phi$ and $\st_{\phi'}$ applied to a divisor of $\C^2$ denote the strict transform of this divisor. These blowups induce a dual picture on the product $X_\Gamma^2$: we illustrate the relation between these blowups and the dual picture in Figure \ref{fig:blowup-square1}.
 
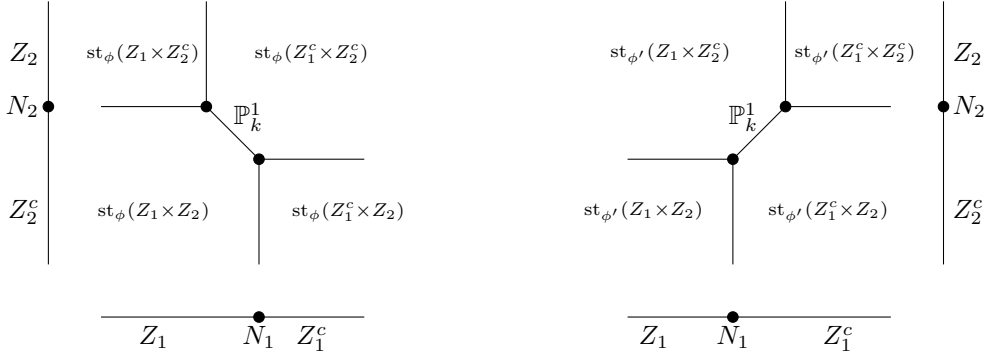
\begin{figure}[ht]
    \centering
    \begin{tikzpicture}[scale=0.7]
    \begin{scope}[shift={(0,0)}]
        \draw (2,1) to (1,2);
        \draw (1,2) to (-1,2);
        \draw (2,1) to (2,-1);
        \draw (1,2) to (1,4);
        \draw (2,1) to (4,1);
        \draw[fill] (2,1) circle [radius=0.1];
        \draw[fill] (1,2) circle [radius=0.1];

        \draw (-2,-1) to (-2,4);
        \draw (-1,-2) to (4,-2);
        \draw[fill] (2,-2) circle [radius=0.1];
        \draw[fill] (-2,2) circle [radius=0.1];
        \node[below] at (0,-2) {$Z_1$};
        \node[below] at (3,-2) {$Z_1^{c}$};
        \node[left] at (-2,3) {$Z_2$};
        \node[left] at (-2,0) {$Z_2^c$};
        \node[left] at (-2,2) {$N_2$};
        \node[below] at (2,-2) {$N_1$};
        
        \node at (3,3) {$\scriptstyle{\st_{\phi}(Z^c_1\times Z^c_2)}$};
        \node at (0, 0) {$\scriptstyle{\st_{\phi}(Z_1\times Z_2)}$};
        \node at (-0.2,3) {$\scriptstyle{\st_\phi(Z_1\times Z^c_2)}$}; 
        \node at (3.7,0) {$\scriptstyle{\st_{\phi}(Z^c_1\times Z_2)}$};
        \node at (1.8,1.8) {$\Ps^1_k$};
    \end{scope}
   \begin{scope}[shift={(10,0)}]
        \draw (1,1) to (2,2);
        \draw (1,1) to (-1,1);
        \draw (1,1) to (1,-1);
        \draw (2,2) to (2,4);
        \draw (2,2) to (4,2);
        \draw[fill] (1,1) circle [radius=0.1];
        \draw[fill] (2,2) circle [radius=0.1];

        \draw (5,-1) to (5,4);
        \draw (-1,-2) to (4,-2);
        \draw[fill] (1,-2) circle [radius=0.1];
        \draw[fill] (5,2) circle [radius=0.1];
        \node[below] at (-0.5,-2) {$Z_1$};
        \node[below] at (3,-2) {$Z_1^{c}$};
        \node[right] at (5,3) {$Z_2$};
        \node[right] at (2,3) {$\scriptstyle{\st_{\phi'}(Z^c_1\times Z^c_2)}$};
        \node[right] at (-1.5,3) {$\scriptstyle{\st_{\phi'}(Z_1\times Z^c_2)}$};
        \node[right] at (5,0) {$Z_2^c$};
        \node[right] at (1.5,0) {$\scriptstyle{\st_{\phi'}(Z^c_1\times Z_2)}$};
        \node[right] at (-2,0) {$\scriptstyle{\st_{\phi'}(Z_1\times Z_2)}$};
        \node[right] at (5,2) {$N_2$};
        \node[below] at (1,-2) {$N_1$};
        
        \node at (1.2,1.8) {$\Ps^1_k$};
    \end{scope}
       
    \end{tikzpicture}

    \caption{The two types of blowup of $\C^2$ around $(N_1,N_2)$.}
    \label{Fig:blowup}
\end{figure}

\begin{Thm}\label{thm:tropgeoAbel}
Let $\pi\col\C\ra B$ be a regualar smoothing of a nodal curve $C$ with smooth components. 
Let $\sigma\col B\ra \C$ be a section of $\pi$ through its smooth locus. Let $\mu$ be a polarization of degree $k$ over the family and $\L$ be an invertible sheaf on $\C$ of degree $k+2$.
Let $(N_1,N_2)$ be a point of $\C^2$, with $N_i$ a node of $C$. Let $Z_1$ and $Z_2$ be subcurves of $C$ such that $N_1\in Z_1\cap Z_1^c$ and $N_2\in Z_2\cap Z_2^c$. Let $e_1$ and $e_2$ be the edges in the dual graph $\G$ of $C$ that correspond to $N_1$ and $N_2$, where $e_i$ is oriented from $Z_i$ to $Z_i^c$. Consider the divisor $\D_{x,y}=\D^\dagger_\L-p_{e_1,x}-p_{e_2,y}$ on $X_\Gamma$, for some $x,y\in [0,1]$. 
%The following properties hold.
%\begin{enumerate}
%    \item If the combinatorial type of $\qs(\D_{x,y})$ is the same as the combinatorial type of $\qs(\D_{x',y'})$ for every $(x,y)$ and $(x',y')$ such that one of the following holds
 %   \begin{enumerate}
 %       \item $0< x< y<1$ and $0<x'<y'<1$,
  %      \item $0<y<x<1$ and $0<y'<x'<1$.
  %      \end{enumerate}
   %     Then the blow up of $\C^2$ at $C_{i_1}\times C_{i_2}$ resolves that Abel map.
%    \item If the combinatorial type of $\qs(\D_{x,y})$ is the same as the combinatorial type of $\qs(\D_{x',y'})$ for every $(x,y)$ and $(x',y')$ such that one of the following hold
 %   \begin{enumerate}
  %      \item   if  $0< x< 1-y<1$ and $0<x'<1-y'<1$, then the combinatorial type of $\qs(\D_{x,y})$ and $\qs(\D_{x',y'})$ coincide.
   %     \item  if   $0<1-y<x<1$ and $0<1-y'<x'<1$,, then the combinatorial type of $\qs(\D_{x,y})$ and $\qs(\D_{x',y'})$ coincide.
       
    %    \end{enumerate}
    %    then the blow up of $\C^2$ at $C_{i_1}\times C_{j_2}$ resolves that Abel map.
%\end{enumerate}
\begin{enumerate}
    \item If the combinatorial type of $\qs(\D_{x,y})$ is constant on the sets 
    \[
    \{(x,y);\ 0<x<y<1\}\text{  and } \{(x,y);\ 0<y<x<1\},
    \]
    then the blowup of $\C^2$ along $Z_1\times Z_2$ resolves the Abel map $\alpha^2_\L$ locally around the point $(N_1,N_2)$.
    \item If the combinatorial type of $\qs(\D_{x,y})$ is constant on the sets 
    \[
    \{(x,y);\ 0<x<1-y<1\}\text{  and } \{(x,y);\ 0<1-y<x<1\},
    \]
    then the blowup of $\C^2$ along $Z_1\times Z_2^c$ resolves the Abel map $\alpha^2_\L$ locally around the point $(N_1,N_2)$.
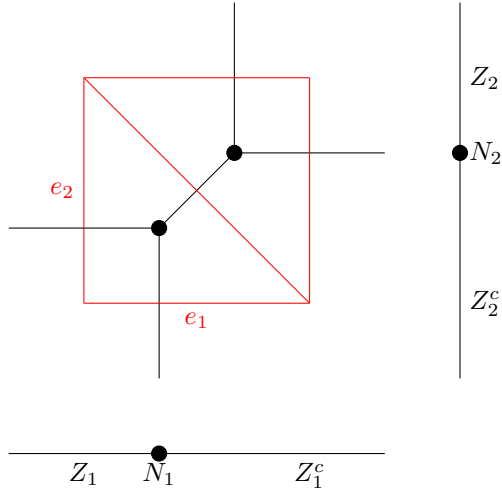
\begin{figure}[ht]
\begin{center}
\begin{tikzpicture}[scale=1]
\begin{scope}[shift={(0,0)}]
\draw (1,1) to (2,2);
\draw (1,1) to (-1,1);
\draw (1,1) to (1,-1);
\draw (2,2) to (2,4);
\draw (2,2) to (4,2);
\draw[fill] (1,1) circle [radius=0.1];
\draw[fill] (2,2) circle [radius=0.1];

\draw (5,-1) to (5,4);
\draw (-1,-2) to (4,-2);
\draw[fill] (1,-2) circle [radius=0.1];
\draw[fill] (5,2) circle [radius=0.1];
\node[below] at (0,-2) {$Z_1$};
\node[below] at (3,-2) {$Z_1^{c}$};
\node[right] at (5,3) {$Z_2$};
\node[right] at (5,0) {$Z_2^c$};
\node[right] at (5,2) {$N_2$};
\node[below] at (1,-2) {$N_1$};

\node[left] at (0,1.5) {\textcolor{red}{$e_2$}};
\node[below] at (1.5,0) {\textcolor{red}{$e_1$}};

\end{scope}

\begin{scope}[shift={(0,0)},scale=6]
\draw[red] (0,0) to (0.5,0); 
\draw[red] (0,0) to (0,0.5); 
\draw[red] (0.5,0) to (0.5,0.5);
\draw[red] (0,0.5) to (0.5,0.5);
\draw[red] (0,0.5) to (0.5,0);
\end{scope}
\end{tikzpicture}
\caption{The sets  $\{(x,y);\ 0<x<1-y<1\}$  and  $\{(x,y);\ 0<1-y<x<1\}$ and the corresponding  blow-up.}
\label{fig:blowup-square1}
\end{center}
\end{figure}
    \item If the combinatorial type of $\qs(\D_{x,y})$ is constant on the set
    \[
    \{(x,y);\ 0<x,y<1\},
    \]
    then the Abel map $\alpha^2_\L$  is defined at the point $(N_1,N_2)$.
        \end{enumerate}
  \end{Thm}
\begin{proof}
Items (1) and (2) follow directly from \cite[Theorem 5.4]{AAP}. 

Let us prove Item (3). Let $\mc N=(N_1,N_2)\in \C^2$.
Let $\phi\col \mc X\ra \mc C^2$ and $\phi'\col\mc Y\ra \mc C^2$ be the blowups respectively along $Z_1\times Z_2$ and $Z_1\times Z_2^c$ (see Figure \ref{Fig:blowup}).
By items (1) and (2), we know that  $\alpha^2_{\mc L}\circ\phi$ and $\alpha^2_{\mc L}\circ\phi'$ are defined respectively over the inverse images $\phi^{-1}(\mc N)\cong \mathbb P^1_k$ and $\phi'^{-1}(\mc N)\cong \mathbb P^1_k$.  
Let $y_0$ be the distinguished point on $\phi^{-1}(\mc N)$ given by 
\[
y_0=\st_\phi(Z_1\times Z_2)\cap \st_\phi(Z_1\times Z_2^c)\cap \st_\phi(Z_1^c\times Z_2^c).
\]
Let $x_1,x_2$  be any two points on $\phi'^{-1}(\mc N)$. We know that $\alpha^2_\L\circ \phi'$ is defined at $x_1$ and $x_2$. For $i=1,2$, consider a map $\rho'_i\col \Spec k[[t]]\ra \mc Y$ such that $\rho'_i(0)=x_i$ and $\rho_i(\eta)$ is contained in $\st_{\phi'}(Z_1\times Z_2^c)$, as in Figure \ref{fig:blowups_phi}, where $0$ and $\eta$ are the special and generic points of $\Spec k[[t]]$, respectively. In particular, we have $\alpha^2_\L\circ \phi'(x_i)=\alpha^2_\L\circ\overline{\rho}_i(0)$, where $\overline{\rho}_i=\phi'\circ\rho'_i\col \Spec(k[[t]])\ra \mc C^2$. By construction, we can lift $\overline{\rho}_i$ to maps $\rho_i\col \Spec k[[t]]\to \mc X$ such that $\rho_1(0)=\rho_2(0)=y_0$. By the same reasoning, we have $\alpha^2_\L\circ \phi(y_0)=\alpha^2_\L\circ\rho_i(0)$ for $i=1,2$.
Then we get: 
\[
\alpha^2_{\mc L}\circ \phi'(x_1)=\alpha^2_{\mc L}\circ \phi(y_0)=\alpha^2_{\mc L}\circ \phi'(x_2).
\]
Hence $\alpha^2_{\mc L}\circ \phi'$ contracts all fibers of $\phi'$. 
Moreover, arguing as in the proof of \cite[Corollary III 11.4]{H}, we have $\phi'_*\mc O_{\mc Y}\cong \mc O_{\mc C^2}$, since $\phi'$ is birational and $\mc C^2$ is normal. 
Hence by the Rigidity Lemma (see \cite[Lemma 1.15, pag.12]{D}) the map 
$\alpha^2_{\mc L}\circ \phi'$ factors through $\phi'$, so $\alpha^2_{\mc L}$ is defined at $\mc N$.

\begin{figure}[htb]
\begin{center}
\begin{tikzpicture}[scale=0.6]
\begin{scope}[shift={(0,0)}]
%\draw (1,2) to (2,1);
\draw (1.5,1.5) to (-1,1.5);
\draw (1.5,1.5) to (1.5,4);
\draw (1.5,1.5) to (1.5,-1);
\draw (1.5,1.5) to (4,1.5);
\draw[fill] (1.5,1.5) circle [radius=0.1];

\draw[->] (-3.5,2.5) to (-1.5,2.5);
\node[above] at (-2.5,2.5) {$\phi$};
\draw[<-] (4,2.5) to (6,2.5);
\node[above] at (5,2.5) {$\phi'$};

\draw[red] (1.5,1.5) to (0,4); 
\draw[red] (1.5,1.5) to (-1,2.5); 

\end{scope}

\begin{scope}[shift={(-8,0)}]
\draw (1,2) to (2,1);
\draw (1,2) to (-1,2);
\draw (1,2) to (1,4);
\draw (2,1) to (2,-1);
\draw (2,1) to (4,1);
\draw[fill] (1,2) circle [radius=0.1];
\draw[fill] (2,1) circle [radius=0.1];

\draw[red] (1,2) to (0,4); 
\draw[red] (1,2) to (-1,3); \node[below] at (1,2) {$y_0$};

\end{scope}

\begin{scope}[shift={(8,0)}]
\draw (1,1) to (2,2);
\draw (1,1) to (-1,1);
\draw (1,1) to (1,-1);
\draw (2,2) to (2,4);
\draw (2,2) to (4,2);
\draw[fill] (1,1) circle [radius=0.1];
\draw[fill] (2,2) circle [radius=0.1];

\draw[red] (1.35,1.35) to (-1,3); 
\draw[red] (1.65,1.65) to (0,4); 
\node[right] at (1.75,1.55) {$x_2$};

\draw[fill,red] (1.35,1.35) circle [radius=0.1];
\draw[fill,red] (1.65,1.65) circle [radius=0.1];
\node[right] at (1.25,1.25) {$x_1$};

\end{scope}
\end{tikzpicture}
\caption{The blowups $\phi$ and $\phi'$.}\label{fig:blowups_phi}
\end{center}
\end{figure}
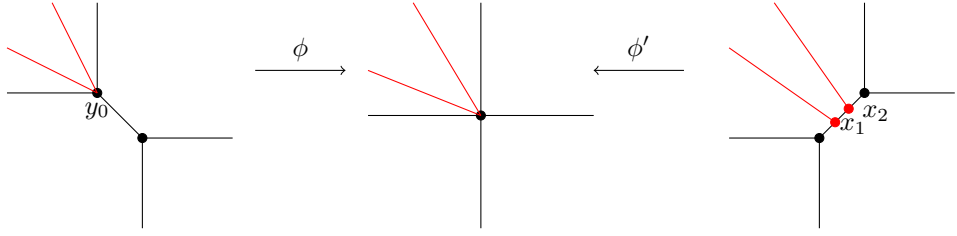
\end{proof}

%\begin{Lem}\label{lem:descend}
%\end{Lem}

\section{The resolution of the degree-2 Abel map}

  \subsection{Local resolutions}

  Throughout this section we will fix a regular smoothing $\pi\col\C\ra B$ of a nodal curve $C$ with a section $\sigma\col B\ra \C$ of $\pi$ through its smooth locus.  
 We will perform blowups of $\mc C^2$ along divisors of type $Z\times Z$, where $Z$ is a subcurve of the special fiber $C$. Actually, we will restrict our attention to a special class of subcurves, called \emph{tails}.   
  
\begin{Def}
A \emph{$\delta$-tail} of a nodal curve $C$ is a connected subcurve $Z$ such that $Z^c$ is connected and $|Z\cap Z^c|=\delta$.
\end{Def}

   \begin{Prop}\label{prop:1-tail}
  Let $\mu$ be a polarization of degree $k$ and $\mc L$ be an invertible sheaf of degree $k+2$  over $\C/B$. Assume that the components of $C$ are smooth.  Consider a point $\mc N=(N_1,N_2)\in\mc C^2$, where $N_1,N_2$ are nodes of $C$, with $N_1=Z\cap Z^{c}$ for a $1$-tail $Z$ of $C$. Then the degree-$2$ Abel map $\alpha^2_{\L}\col\C^{2}\dashrightarrow \overline{\J}_{\mu}^\sigma$ is defined at $\mc N$.
  \end{Prop}

  \begin{proof}
  Let $\Gamma$ be the dual graph of $C$ and $X=X_\Gamma$ the associated tropical curve with edges of unitary lengths.  
  %Let $(X,p_0)$ be the pointed tropical curve associated to $\pi$. We denote by $\G$ the dual graph of $C$ (so that $X=X_\Gamma$), 
  Let $v_0$ be the vertex of $\Gamma$ corresponding to $P_0=\sigma(0)$, and $p_0\in X$ be the point corresponding to $v_0$. We let $e_1$ and $e_2$ be the edges of $\Gamma$ corresponding to $N_1$ and $N_2$. The tropical Abel map $\alpha^\trop_{2,\D^\dagger_{\mathcal L}}\col X^2\ra J^\trop_{p_0,\mu}$ takes a pair $(p_{e_1,t_1},p_{e_2,t_2})$, for real numbers $t_1,t_2\in (0,1)$, to the class of the divisor on $X$ given by:
\begin{equation}\label{eq:alpha2}
\alpha^\trop_{2,\D^\dagger_{\mathcal  L}}(p_{e_1,t_1},p_{e_2,t_2})=[\mc D^\dagger_\L-p_{e_1,t_1}-p_{e_2,t_2}].
\end{equation}
 We define the divisor  
 $\mc P=p_{e_1,t_1}-p_{e_1,0}$ on $X$.
Since $N_1=Z\cap Z^c$ for a $1$-tail $Z$ of $X$, we have that the graph obtained from $\G$ by removing the edge $e_1$ is not connected. Hence the divisor $\mc P$ on $X$ is principal. So we can write:
\[
\alpha^{\trop}_{2,\D^\dagger_{\mathcal  L}}(p_{e_1,t_1},p_{e_2,t_2})=[\mc{\wh D}^\dagger-p_{e_2,t_2}],
\]
where $\mc{\wh D}^\dagger=\mc D^\dagger_\L-p_{e_1,0}$ (which is a divisor on $X$ induced by a divisor on $\Gamma$). So we reduce ourselves to the case of the degree-$1$ Abel map. As explained in \cite[Lemma 5.10]{AAP} and in the proof of \cite[Theorem 5.8]{AAP}, the combinatorial type of the quasistable divisor on $X$ equivalent to $\mc{\wh D}^\dagger-p_{e_1,t_2}$ is independent of $t_2$. Hence the combinatorial type of the quasistable divisor on $X$ equivalent to $\mc D^\dagger_\L-p_{e_1,t_1}-p_{e_2,t_2}$ is independent of the pairs $t_1,t_2\in (0,1)$. 
By Theorem \ref{thm:tropgeoAbel} (3), we deduce that the Abel map $\alpha^2_{\mc L}$ is already defined at $(N_1,N_2)$.
  \end{proof}

   \begin{Prop}\label{prop:2-tail}
   Let $\mu$ be a polarization of degree $k$ and $\mc L$ an invertible sheaf of degree $k+2$  over $\C/B$. Assume that the components of $C$ are smooth. 
  Let $Z$ be a $2$-tail of $C$ and write $\{N_1,N_2\}=Z\cap Z^c$. Consider the point 
  \[
  \mc N=(N_1,N_2)\in (Z\cap Z^c)\times (Z\cap Z^c)\subset \mc C^2.
  \]
  Let $\phi\col\wt{\mc C}^2\ra \C^2$ be the blowup of $\C^2$ with center $Z\times Z$. Then the rational map 
  \[
 \wt{\alpha}^2_{\L}\col \wt{\mc C}^2\stackrel{\phi}{\longrightarrow} \mc C^2\stackrel{\alpha^2_{\mc L}}{\dashrightarrow}\overline{\mc J}^\sigma_\mu
  \]
is defined along the rational curve $\phi^{-1}(\mc N)\cong\mathbb P^1_k$.
 \end{Prop}
  
  \begin{proof}
  We can keep the set-up of Proposition \ref{prop:1-tail}.  
The tropical Abel map  $\alpha^\trop_{2,\D^\dagger_{\mathcal L}}\col X^2\ra J^\trop_{p_0,\mu}$ is as in Equation \eqref{eq:alpha2}. 
Assume that $t_1>t_2$. We define the divisor on $X$:
\[
\mc P=p_{e_1,t_1}-p_{e_1,t_1-t_2}+p_{e_2,t_2}-p_{e_2,0}.
\]
Since $\{N_1,N_2\}=Z\cap Z^c$ for a $2$-tail $Z$ of $C$, we have that the graph obtained from $\G$ by removing the edges $e_1,e_2$ is not connected. Hence  $\mc P$ is a principal divisor. Then we have
\[
\alpha^{\trop}_{2,\D^\dagger_{\mathcal  L}}(p_{e_1,t_1},p_{e_2,t_2})=[\mc{\wh D}^\dagger-p_{e_1,t}],
\]
where $t=t_1-t_2$ and $\mc{\wh D}^\dagger=\mc D^\dagger_\L-p_{e_2,0}$ (which is a divisor induced by a divisor on $\Gamma$). So we reduce ourselves to the case of the degree-$1$ Abel map: as explained in \cite[Lemma 5.10]{AAP} and in the proof of \cite[Theorem 5.8]{AAP}, the combinatorial type of the quasistable divisor on $X$ equivalent to $\mc{\wh D}^\dagger-p_{e_1,t}$ is independent of $t$. Hence the combinatorial type of the quasistable divisor on $X$ equivalent to $\mc D^\dagger_\L-p_{e_1,t_1}-p_{e_2,t_2}$ is independent of $(t_1,t_2)$ whenever $t_1>t_2$. A similar reasoning can be done for the case $t_1<t_2$. Hence, using Theorem \ref{thm:tropgeoAbel} (1), we conclude that the blowup along $Z\times Z$ gives rise to a resolution of $\alpha^2_{\mc L}$ locally around $\mc N=(N_1,N_2)$.
\end{proof}
  
  \begin{Prop}\label{prop:diagonal}
   Let $\mu$ be a polarization of degree $k$ and $\mc L$ an invertible sheaf of degree $k+2$  over $\C/B$. Assume that the components of $C$ are smooth.  
   Consider the point 
  $\mc N=(N,N) \in \mc C^2$, for a node $N$ of $C$.
  Let $\phi\col\wt{\mc C}^2\ra \C^2$ be the blowup of $\C^2$ with center the diagonal subscheme of $\C^2$. Then the rational map 
  \[
 \wt{\alpha}^2_{\L}\col \wt{\mc C}^2\stackrel{\phi}{\longrightarrow} \mc C^2\stackrel{\alpha^2_{\mc L}}{\dashrightarrow}\overline{\mc J}^\sigma_\mu
  \]
is defined along the rational curve $\phi^{-1}(\mc N)\cong\mathbb P^1_k$.
  \end{Prop}

  \begin{proof}
     We can keep the set-up of Proposition \ref{prop:1-tail}. The tropical Abel map  $\alpha^\trop_{2,\D^\dagger_{\mathcal L}}\col X^2\ra J^\trop_{p_0,\mu}$  is as in Equation \eqref{eq:alpha2}, where $e:=e_1=e_2$.
Assume that $t_1+t_2<1$. We define the principal divisor on $X$:
\[
\mc P=p_{e,0}-p_{e,t_1}-p_{e,t_2}+p_{e,t_1+t_2}.
\]
Then we have
\[
\alpha^{\trop}_{2,\D^\dagger_{\mathcal  L}}(p_{e,t_1},p_{e,t_2})=[\mc{\wh D}^\dagger-p_{e,t}],
\]
where $\mc{\wh D}^\dagger=\mc D^\dagger_{\mc L}-p_{e,0}$ and $t=t_1+t_2$.   So we reduce ourselves to the case of the degree-$1$ Abel map: as explained in \cite[Lemma 5.10]{AAP} and in the proof of \cite[Theorem 5.8]{AAP}, the combinatorial type of the quasistable divisor on $X$ equivalent to $\mc{\wh D}^\dagger-p_{e,t}$ is independent of $t$. Hence the combinatorial type of the quasistable divisor on $X$ equivalent to $\mc D^\dagger_\L-p_{e,t_1}-p_{e,t_2}$ is independent of $(t_1,t_2)$ whenever $t_1+t_2<1$.\par
 The reasoning is similar for $t_1+t_2>1$: we just consider 
 \[
 \mc P=p_{e,1}-p_{e,t_1}-p_{e,t_2}+p_{e,t_1+t_2-1},
 \]
 $\mc{\wh D}^\dagger=\mc D^\dagger_{\mc L}-p_{e,1}$ and  $t=t_1+t_2-1$, so that 
 \[
\alpha^{\trop}_{2,\D^\dagger_{\mathcal  L}}(p_{e,t_1},p_{e,t_2})=[\mc{\wh D}^\dagger-p_{e,t}].
\]
%where $\mc{\wh D}^\dagger=\mc D^\dagger_{\mc L}-p_{e,1}$ and. 
By Theorem \ref{thm:tropgeoAbel} (2), we deduce that the blowup along the diagonal subscheme of $\C^2$ gives rise to a resolution of $\alpha^2_{\mc L}\circ \phi$ locally around $\mc N=(N,N)$.
  \end{proof}
  
  \subsection{The resolution of the Abel-Jacobi map}
  
  Our main goal is to give a complete resolution of the degree-$2$ Abel-Jacobi map of any nodal curve, namely the map taking a pair $(Q_1,Q_2)$ of points on a curve $C$ to $\mathcal O_C(2P_0-Q_1-Q_2)$ for a given smooth point $P_0$ of $C$. This is done in Theorem \ref{thm:Abel2}. Before, we need two results. 
  
  \begin{Lem}\label{lem:perfect}
  Let $Z$ be a $\delta$-tail of a curve $C$.
\begin{itemize}
    \item[(1)]  If $Z\cap Z^c\subset Z'$ for some tail $Z'$ of $C$, then either $Z\subset Z'$ or $Z^c\subset Z'$.
    \item[(2)] If $|(Z\cap Z^c \cap Z'\cap (Z')^c|=\delta-1$ for some tail $Z'$ of $C$, then one of the following conditions holds:
    \[
    Z\subset Z', \;\; Z'\subset Z, \;\; Z^c\subset Z',\;\; Z'\subset Z^c.
    \]
\end{itemize} 
  \end{Lem}
  
  \begin{proof}
  See \cite[Lemma 2.4]{P}.
  \end{proof}

  \begin{Lem}
  Let $P_0$ be a smooth point of $C$.  
 Let $\mc T=(Z_1,\dots,Z_h)$ be a sequence of tails of $C$, where $Z_i$ is a $k_i$-tail with $k_i\in\{2,3\}$ and $P_0\not\in Z_i$. Consider the sequence of blowups
 \[
\phi_{\mc T}\col\wt\C^2_h\stackrel{\phi_h}{\lra}\cdots \stackrel{\phi_3}{\lra}\wt\C^2_2\stackrel{\phi_2}{\lra}\wt\C^2_1\stackrel{\phi_1}{\lra}\wt\C^2_0\stackrel{\phi_0}{\lra}\C^2
  \]
 where $\phi_0$ is the blowup of $\C^2$ along its diagonal subscheme and $\phi_i$ is the blowup of $\wt{\C}^2_{i-1}$ along the strict transform of the divisor $Z_i\times Z_i$ of $\C^2$ via $\phi_1\circ\cdots\circ \phi_{i-1}$. Then $\phi_{\mc T}$ is independent of the ordering of the sequence $\mc T$.
\end{Lem}
  
  \begin{proof}
  Assume that a permutation $\T'$ of $\T$ gives rise to a blowup $\phi_{\T'}$ different from $\phi_{\T}$. This implies that, locally at a point $\mc N=(N_1,N_2)$ with $N_1,N_2$ distinct nodes of $C$, the blowups $\phi_{\T}$ and $\phi_{\T'}$ are different.
  We can assume that locally at $\mc N$, the blowup $\phi_{\T}$ has center $Z_i\times Z_i$ and the blowup $\phi_{\mc T'}$ has center $Z_j\times Z_j$, for $i,j\in\{1,\dots,h\}$, so that we have $Z_i\ne Z_j$ and
  \[
  \{N_1,N_2\}\subset Z_i\cap Z_i^c\cap Z_j\cap Z_j^c.
  \]
  By Lemma \ref{lem:perfect}, one of the following conditions holds:
  \begin{equation}\label{eq:perfect}
  Z_i\subset Z_j, \;\; Z_j\subset Z_i, \;\; Z^c_i\subset Z_j, \;\; Z_j\subset Z_i^c.
  \end{equation}
  Let $C_1$, $C'_1$, $C_2$, $C'_2$ be the components of $C$ such that $N_1\in C_1\cap C'_1$ and $N_2\in C_2\cap C'_2$. Since $\phi_{\T}$ and $\phi_{\mc T'}$ are different locally at $\mc N$, we can assume, without loss of generality, that $C_1\cup C_2\subset Z_i$ and $C_1\cup C_2'\subset Z_j$.
  %(see Figure \ref{Fig:blowup}). 
  Then we have 
  \[
  C_1\subset Z_i\cap Z_j, \;\;  C_2\subset Z_i\cap Z_j^c, \;\; C'_2\subset Z_i^c\cap Z_j.
  \]
   On the other hand:
   \begin{enumerate}
       \item since $C_1\subset Z_i\cap Z_j$, it follows that $Z_j\not \subset Z_i^c$.
       \item since $C_2\subset Z_i\cap Z_j^c$, it follows that $Z_i\not\subset Z_j$.
       \item since $C_2'\subset Z_i^c\cap Z_j$, it follows that $Z_j\not\subset Z_i$.
       \item since  $P_0\in Z_i^c\setminus Z_j$, it follows that $Z_i^c\not\subset Z_j$.
         \end{enumerate} 
           This contradicts Equation \eqref{eq:perfect}.
  \end{proof}

  \begin{Thm}[Degree-$2$ Abel-Jacobi map]\label{thm:Abel2}
  Let $\pi\col\C\ra B$ be a regular smoothing of a nodal curve $C$. Let $\sigma\col B\ra \C$ be a section of $\pi$ through its smooth locus and $\mu$ be the trivial degree-$0$ polarization. Let $Z_1,\dots,Z_N$ be the 2-tails and the 3-tails of $C$ which do not contain $\sigma(0)$. Consider the sequence of blowups
  \[
\wt\C^2_N\stackrel{\phi_N}{\lra}\cdots \stackrel{\phi_2}{\lra}\wt\C^2_1\stackrel{\phi_1}{\lra}\wt\C^2_0\stackrel{\phi_0}{\lra}\C^2
  \]
  where $\phi_0$ is the blowup of $\C^2$ along its diagonal subscheme and $\phi_i$ is the blowup of $\wt{\C}^2_{i-1}$ along the strict transform of the divisor $Z_i\times Z_i$ of $\C^2$ via $\phi_0\circ\cdots\circ \phi_{i-1}$. Then the rational map
  \[
\alpha^2_{\mc O_\C}\circ\phi_0\circ\cdots\circ\phi_N\col \wh\C^2_N\dra\overline{\mc J}^\sigma_\mu
  \]
  is a morphism, i.e., it is defined everywhere.
  \end{Thm}

Before proving the theorem, we need to recall a result in \cite{P} describing how to convert the sheaf $\mc O_C(2P_0-Q_1-Q_2)$ into a $(\sigma,\mu)$-quasistable sheaf, where $P_0,Q_1,Q_2$ are smooth points of  the nodal curve $C$. We will give the graph-theoretical equivalent of this result, which suits better with our purposes. More precisely, given a graph $\G$ and vertices $v_0,v_1,v_2\in V(\G)$, we will describe the $(v_0,\mu)$-quasistable divisor on $\G$ equivalent to $2v_0-v_1-v_2$ (see Theorem \ref{thm:convert}).

Let $\Gamma$ be a graph. Given a subset $V\subset V(\Gamma)$, we denote by $\Gamma(V)$ the subgraph of $\Gamma$ whose set of vertices is $V$ and whose edges are the edges of $\Gamma$ connecting two (possibly coinciding) vertices of $V$. 

\begin{Def}
A \emph{hemisphere} of $\Gamma$ is a subset $H\subset V(\Gamma)$ such that $\Gamma(H)$ and $\Gamma(H^c)$ are connected subgraphs of $\Gamma$. A $\delta$-hemisphere of $\Gamma$ is a hemisphere $H$ such that $|E(H,H^c)|=\delta$. 
\end{Def}

 We denote by $\mc H_{\Gamma,\delta}$ the set of $\delta$-hemispheres of $\Gamma$. Given subsets $V,W\subset V(\Gamma)$, we define:
\[
\mc H_{\Gamma,\delta}(V,W):=\{H \in \mathcal H_{\Gamma,\delta} | V\subset H^c \text{ and } W\subset H\}.
\]

Let $S$ be a finite set. We say that a set $\mc H$ of subsets of $S$ is union-closed (respectively, intersection-closed) if  $H_1\cup H_2\in \mc H$ (respectively, $H_1\cap H_2\in \mc H$) for every $H_1,H_2\in \mc H$. We note that every non-empty intersection-closed set has a unique minimal element and every non-empty union-closed set has a unique maximal element.

\begin{Prop}
\label{prop:closed}
Let $\Gamma$ be a graph and $v_0,v_1,v_2$ vertices of $\Gamma$.
Then the sets $\mc H_{\Gamma,1}(v_0,v_1)$ and $\mc H_{\Gamma,2}(v_0,\{v_1,v_2\})$ are union-closed and intersection-closed.
\end{Prop}

\begin{proof}
See \cite[Lemma 4.3]{CE} and \cite[Section 3 and Proposition 3.1]{P}.
\end{proof}

\begin{Def}
Given subsets $V,W\subset V(\Gamma)$, we say that $W$ is \emph{$V$-free} if $E(V,V^c)\cap E(W,W^c)=\emptyset$. 
\end{Def}

\begin{Rem}\label{rem:H1H2}
A $1$-hemisphere $H$ is $H'$-free for every $\delta$-hemisphere $H'\ne H, H^c$.
 If $H_1,H_2\subset V(\Gamma)$ are $V$-free hemispheres, for some $V\subset V(\G)$, then $H_1\cap H_2$ and $H_1\cup H_2$ are also $V$-free.
\end{Rem}

\begin{Def}
Let $\G'$ be a subdivision of $\G$. Let $V$ be a subset of $V(\Gamma')$. We say that an edge $e\in E(\Gamma)$ is \emph{fully contained} in $V$ if $V$ contains the vertices incident to $e$ and all the vertices over $e$. Note that when $\Gamma'=\Gamma$, this simply means that $e\in E(V,V)$.
\end{Def}

\begin{Lem}\label{lem:intersect23}
  Let $H_2$ and $H_3$ be a $2$-hemisphere and a $3$-hemisphere. Write $E(H_2,H_2^c)=\{f_1,f_2\}$ and $E(H_3,H_3^c)=\{e_1,e_2,e_3\}$. Assume that the intersection $H=H_2\cap H_3$ is non-empty and properly contained in $H_2$ and $H_3$. Assume that $H_2\cup H_3\neq V(\Gamma)$. Then, up to reordering the indices, one of the following  properties hold
    \begin{enumerate}
      %  \item $H=\emptyset$ 
       % \item $H_2\cup H_3=V(\Gamma)$
    %    \item $H=H_2$
     %   \item $H=H_3$
        \item $H$ is a $2$-hemisphere such that  $E(H,H^c)=\{f_1,e_1\}$, with $f_1$ fully contained in $H_3$ and $e_1$ fully contained in $H_2$, while $f_2$ is fully contained in $H_3^c$ and $e_2,e_3$ are fully contained in $H_2^c$.
        
        \item $H$ is a $3$-hemisphere such that $E(H,H^c)=\{f_1,e_1,e_2\}$, with $f_1$ fully contained in $H_3$ and $e_1,e_2$ fully contained in $H_2$, while $f_2$ is fully contained in $H_3^c$ and $e_3$ is fully contained in $H_2^c$.
    \end{enumerate}
\end{Lem}
\begin{proof}
By the hypothesis, the sets $H$, $H_2\setminus H$, $H_3\setminus H$ and $H_2^c\cap H_3^c$ are nonempty and form a partition of $V(\Gamma)$. Since $H_3$ is connected, we have that $E(H,H_3\setminus H)$ is nonempty. However $E(H,H_3\setminus H)\subseteq E(H_2,H_2^c)=\{f_1,f_2\}$. We assume  that $f_1\in E(H,H_3\setminus H)$. Arguing in a similar manner, using that $H_3^c$ is connected, we have that $f_2\in E(H_2\setminus H, H_2^c\cap H_3^c)$. Even more, we can conclude that $e_1\in E(H,H_2\setminus H)$ and $e_3\in E(H_3\setminus H,H_2^c\cap H_3^c)$ (see Figure \ref{fig:intersection_hemispheres}).
  \begin{figure}[ht]
    \centering
    \begin{tikzpicture}[scale=1.7]
    \draw  (0,0) circle (0.5cm);
    \node at (0,0) {$H$};
    \draw (2,0) circle (0.5cm);
    \node at (2,0) {$H_2\setminus H$};
    \draw (0,-2) circle (0.5cm);
    \node at (0,-2) {$H_3\setminus H$};
    \draw (2,-2) circle (0.5cm);
    \node at (2,-2) {$H_2^c\cap H_3^c$};
    \draw[fill] (0.4,0) circle (0.02cm);
    \draw[fill] (1.6,0) circle (0.02cm);
    \draw  (0.4,0) -- (1.6,0);
    \node[above] at (1,0) {$e_1$};
    \draw[fill] (0.4,-2) circle (0.02cm);
    \draw[fill] (1.6,-2) circle (0.02cm);
    \draw  (0.4,-2) -- (1.6,-2);
    \node[above] at (1,-2) {$e_3$};
    \draw[fill] (0,-0.4) circle (0.02cm);
    \draw[fill] (0,-1.6) circle (0.02cm);
    \draw  (0,-0.4) -- (0,-1.6);
    \node[left] at (0,-1) {$f_1$};
    \draw[fill] (2,-0.4) circle (0.02cm);
    \draw[fill] (2,-1.6) circle (0.02cm);
    \draw  (2,-0.4) -- (2,-1.6);
    \node[left] at (2,-1) {$f_2$};
    \end{tikzpicture}
    \caption{The edges $e_1,e_3,f_1,f_2$.}
    \label{fig:intersection_hemispheres}
\end{figure}
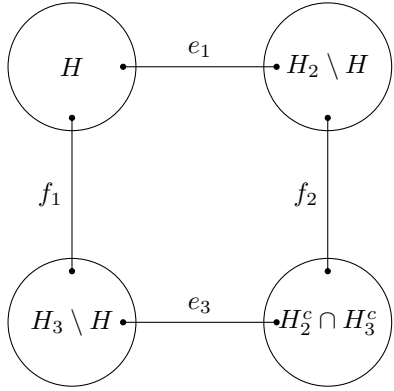

Since $H_2$ is a $2$-hemisphere, we have that 
  \begin{align*}
  \quad\quad\quad\quad\quad\quad\quad\quad E(H,H_3\setminus H)&=\{f_1\}, & E(H_2\setminus H,H_3\setminus H)&=\emptyset,\quad\quad\quad\quad\quad\quad\quad\quad \\
  \quad\quad\quad\quad\quad\quad\quad\quad E(H,H_2^c\cap H_3^c)&=\emptyset, & \quad E(H_2\setminus H,H_2^c\cap H_3^c)&=\{f_2\}.\quad\quad\quad\quad\quad\quad\quad\quad
  \end{align*}
  Hence $E(H_3,H_3^c)=E(H,H_2\setminus H)\cup E(H_3\setminus H, H_2^c\cap H_3^c)$. Therefore, the edge $e_2$ only has two possibilities: it belongs to either $E(H_3\setminus H, H_2^c\cap H_3^c)$ or $E(H,H_2\setminus H)$. In the former case $H$ satisfies the conditions in item (1), while in the latter case it satisfies the conditions in item (2).
  
  Note that the sets $H$, $H_3\setminus H$, $H_2\setminus H$ and $H_2^c\cap H_3^c$ are connected because $H_3=H\cup (H_3\setminus H)$ is connected and $H$ and $H_3\setminus H$ are connected by a single edge, hence each $H$ and $H_3\setminus H$ must be connected. The same reasoning holds for $H_2\setminus H$ and $H_2^c\cap H_3^c$, using the fact that $H_3^c$ is connected. Hence $H^c=(H_3\setminus H) \cup (H_2\setminus H) \cup (H_2^c\cap H_3^c)$ is connected, which means that $H$ is a hemisphere.
\end{proof}

We let $H_{2,1}$ be the minimal element of $\mc H_{\Gamma,2}(v_0,\{v_1,v_2\})$ (which exists and is unique by Proposition \ref{prop:closed}).  Define 
\[
\mc H^{\free}_{\Gamma,2}(v_0,\{v_1,v_2\})=\{H_{2,1},\dots,H_{2,m_2}\},
\]
where $H_{2,i}$ is the minimal element of the set of hemispheres of $\mc H_{\Gamma,2}(v_0,\{v_1,v_2\})$ that are $H_{2,j}$-free for every $j<i\le m_2$ and containing $H_{2,i-1}$. The hemisphere $H_{2,i}$ exists and is unique since the set 
\[
\{H\in \mc H_{\Gamma,2}(v_0,\{v_1,v_2\})| H_{2,i-1}\subset H, H\text{ is } H_{2,j}\text{-free for $j=1,\ldots, i-1$}\}
\]
is intersection-closed by Proposition \ref{prop:closed} and Remark \ref{rem:H1H2}.

Notice that we have a sequence of nested $2$-hemispheres
\[
H_{2,1}\subset H_{2,2}\subset \ldots \subset H_{2,m_2}.
\]

We let $\mc H'_{\Gamma,3}(v_0,\{v_1,v_2\})$ be the subset of  $\mc H_{\Gamma,3}(v_0,\{v_1,v_2\})$ of the hemispheres that are $H$-free for every $H\in \mc H_2^{\free}(v_0,\{v_1,v_2\})$.

\begin{Prop}
The subset $\mc H'_{\Gamma,3}(v_0,\{v_1,v_2\})$ is intersection-closed.
\end{Prop}

\begin{proof}
See  \cite[Proposition 3.5]{P}.
\end{proof}

We let $H_{3,1}$ be the minimal element of $\mc H'_{\Gamma,3}(v_0,\{v_1,v_2\})$  and define 
\[
\mc H^{\free}_{\Gamma,3}(v_0,\{v_1,v_2\})=\{H_{3,1},\dots,H_{3,m_3}\},
\]
where $H_{3,i}$ is the minimal element of the set of hemispheres $\mc H'_{\Gamma,3}(v_0,\{v_1,v_2\})$ that are $H_{3,j}$-free for every $j<i\le m_3$ and containing $H_{3,i-1}$. As before, we have that $\mc H_{\Gamma,3}^{\free}$ is well-defined. Notice that we have a sequence of nested $3$-hemispheres
\[
H_{3,1}\subset H_{3,2}\subset\dots\subset H_{3,m_3}.
\]

%\begin{Lem}
%Let $\Gamma$ be a graph and $\widehat{\Gamma}$ be a refinement of $\Gamma$. Let $v_0,v_1,v_2$ be vertices of $\Gamma$ (which are vertices of $\widehat{\Gamma}$ as well). Then for every $k\in\{2,3\}$, 
%\[
%\widehat{H_1}^k=H_1^k\cup \bigcup_{e\in E(H_1^k)}\{ \widehat{v}; \widehat{v}\text{ is an exceptional vertex over e}\},
%\]
%in other words, we have a natural bijection $E(\widehat{H}_1^k,(\widehat{H}_1^k)^c)\leftrightarrow E(H_1^k,(H_1^k)^c)$ that takes an edge $\widehat{e}$ to the edge ($G(\widehat{e})$) $e$ such that $\widehat{e}$ is over $e$.
%\end{Lem}
%\begin{proof}
% The result follows from the following claim: if $\widehat{H}$ is a $\delta$-hemisphere on $\widehat{\Gamma}$ containing at least one vertex in $V(\Gamma)$, then $H:=\widehat{H}\cap V(\Gamma)$ is a $\delta$-hemisphere on $\Gamma$ or $H=V(\Gamma)$ and $\delta=2$. Since $\widehat{\Gamma}(\widehat{H})$ and $\widehat{\Gamma}(\widehat{H}^c)$ are connected we have that if an edge $e\in E(\Gamma)$ has endpoints  $v_1,v_2\in H$ and  $\widehat{v}_1,\ldots, \widehat{v}_m$ are the vertices in $V(\widehat{\Gamma})$ over $e$, then either $\widehat{v}_i\in \widehat{H}$, or $\widehat{H}^c\subset \{\widehat{v}_1,\ldots, \widehat{v}_m\}$. If the latter, holds, then $H=V(\Gamma)$ and $\delta=2$. So, we assume that if \par
  %  WeIf we have the former, this means that if $\widehat{e}\in E(\widehat{H},\widehat{H}^c)$, then $G(\widehat{e})\in E(H,H^c)$
%\end{proof}

\begin{Rem}\label{rem:orientation}
  Let $k=2,3$. Notice that we have a natural orientation on every edge $e\in E(H_{k,i}, H_{k,i}^c)$ such that $s(e)\in H_{k,i}$ and $t(e)\in H_{k,i}^c$. 
  Moreover, if $e\in E(H_{k,i}, H_{k,i}^c)$, then $t(e)\in H_{k,i+1}$.
\end{Rem}

Finally we set:
\begin{align*}
\mc F_\Gamma(v_0,v_1,v_2)=&\ \mc H_{\Gamma,1}(v_0,v_1)\ \sqcup \ \mc H_{\Gamma,1}(v_0,v_2)\\ &\sqcup\ \mc H^{\free}_{\Gamma,2}(v_0,\{v_1,v_2\})\ \sqcup \ \mc H^{\free}_{\Gamma,3}(v_0,\{v_1,v_2\}).
\end{align*}
Notice that the same $1$-hemisphere could belong in both $\mc H_{\Gamma,1}(v_0,v_1)$ and $\mc H_{\Gamma,1}(v_0,v_2)$.

 For every subset $V\subset V(\Gamma)$, we let $\div(V)$ be the principal divisor on $\G$ given by:
\[
\div(V)=\sum_{e\in E(V,V^c)}(s(e)-t(e)), 
\]
where the orientation is chosen such that $s(e)\in V$ for every $e\in E(V,V^c)$. Let $v_0,v_1,v_2$ be vertices of $\Gamma$ and $\mu$ the trivial degree-$0$ polarization on $\Gamma$.
The following result tells us how to find the $(v_0, \mu)$-quasistable divisor equivalent to the divisor $2v_0-v_1-v_2$. %Recall that $\div(W)$ is a principal divisor on $\Gamma$, 

\begin{Thm}\label{thm:convert}
Let $\Gamma$ be a graph and $v_0,v_1,v_2$ vertices of $\Gamma$. Then
\[
2v_0-v_1-v_2-\sum_{V\in \mc F_\Gamma(v_0,v_1,v_2)}\div(V)
\]
is the $(v_0,\mu)$-quasistable divisor equivalent to $2v_0-v_1-v_2$.
\end{Thm}

\begin{proof}
See \cite[Theorem 5.3]{P}.
\end{proof}

Before going on with the proof of Theorem \ref{thm:Abel2}, we need to introduce a divisor on a tropical curve $X$ attached to a hemisphere of its underlying graph. 

\begin{Def}
Let $X$ be a tropical curve and 
 $(\Gamma,\ell)$ be a model of $X$. Let $v_0$ be a vertex of $\Gamma$. 
For a hemisphere $H$ of $\Gamma$, consider the orientation on an edge $e\in E(H,H^c)$ from the vertex of $e$ contained in $H$ to the vertex of $e$ contained in $H^c$ (see Remark \ref{rem:orientation}). 
%Given a hemisphere $H$ of $\wh\Gamma$,
%such that $\wh\ell(e)=\wh\ell(e')$ for any two edges $e,e'\in E(H,H^c)$, 
We define the divisor 
\[
\mc P_H=\sum_{e\in E(H,H^c)} p_{e,0}-\sum_{e\in E(H,H^c)} p_{e,\ell(e)}. 
\]
%It is easy to check that 
\end{Def}

\begin{Rem}\label{rem:P_principal}
  Notice that if $\ell(e)=\ell(e')$ for every $e,e'\in E(H,H^c)$, then $\mc P_H$ is a principal divisor on $X$. 
\end{Rem}

\begin{proof}[Proof of Theorem \ref{thm:Abel2}]

We can assume that all the components of $C$ are smooth. Indeed, the general case follows from the case in which the components of $C$ are smooth arguing as in the last part of the proof of \cite[Theorem 5.8]{AAP}, and using \cite[Theorem 1.3]{P}.

 Since quasistability is an open property by \cite[Proposition 34]{EE01}, it is enough to check that the global blowup of $\C^2$ described in the statement is a blowup resolving the Abel map $\alpha^2_{\mc O_C}\col \mc C^2\ra \overline{\mc J}^\sigma_\mu$ locally around any point $\mc N=(N_1,N_2)$ of $\mc C^2$ for $N_i$ a node of $C$. 
 
  If $N_1=Z\cap Z^c$ for a $1$-tail $Z$ of $C$, the result follows from Proposition \ref{prop:1-tail}.  If $N_1=N_2$, the result follows from Proposition \ref{prop:diagonal}.  So we will assume, through the rest of the proof, that $N_1\neq N_2$  and neither $N_1$ nor $N_2$ disconnects $C$.

%In the remaining cases, we use Theorem \ref{thm:tropgeoAbel}.

We will use Theorem \ref{thm:tropgeoAbel}. Let $\G_0$ be the dual graph of $C$, and let
   $X=X_{\G_0}$, namely, the tropical curve whose underlying graph is $\Gamma_0$  with all unitary lengths.  
  Let $v_0$ be the vertex of $\Gamma_0$  corresponding to $P_0=\sigma(0)$, and $p_0\in X$ be the point corresponding to $v_0$.
%Let $(X,p_0)$ be the pointed tropical curve associated to $\pi$. We denote by $\G$ the dual graph of $C$ (so that $X=X_\Gamma$), and $v_0$ the vertex of $\Gamma$ corresponding to $p_0$. 
We have $\D^\dagger_{\mathcal O_\C}=2p_0$.
The tropical Abel map $\alpha^\trop_{2,\D^\dagger_{\mathcal O_C}}\col X^2\ra J^\trop_{p_0,\mu}$ takes a pair $(p_{e_1,t_1},p_{e_2,t_2})$, for edges $e_1,e_2\in E(\Gamma_0)$ and real numbers $t_1,t_2\in [0,1]$, to:
\[
\alpha^\trop_{2,\D^\dagger_{\mathcal O_C}}(p_{e_1,t_1},p_{e_2,t_2})=[2p_0-p_{e_1,t_1}-p_{e_2,t_2}].
\]
 
%Since we deal with the Abel map of degree $2$, the hypercubes $\mc{H}_{f}$ are squares of type $e_1\times e_2$ for edges $e_1,e_2\in E(\Gamma)$. There are two unimodular triangulations of the square $e_1\times e_2$: as illustrated in Figure \ref{Fig:blowup-square}, they are ``dual" to the two blowups in Figure \ref{Fig:blowup}.\footnote{Antecipar a figura}
Let $(\Gamma,\ell)$ be the model of $X$ such that $\Gamma$ is the refinement of $\Gamma_0$ obtained by inserting vertices $v_{e_1},v_{e_2}$ in the interior of $e_1,e_2$, respectively, with $\ell([s(e_i),v_{e_i}])=t_i$. 
 %Let $\Gamma$ be the dual graph of $C$ and $X=X_\Gamma$ the associated tropical curve with edges of unitary lengths.  
 % Let $v_0$ be the vertex of $\Gamma$ corresponding to $P_0=\sigma(0)$, and $p_0\in X$ be the point corresponding to $v_0$.
%Let $\wh{\Gamma} $ be the refinement of $\Gamma$ with $V(\wh\Gamma)=V(\Gamma)\cup\{v_{e_1},v_{e_2}\}$ where $v_{e_i}$ is over $e_i$.
We let $K_{2,1}$ be the minimal element of $\mathcal H^{\free}_{\Gamma,2}(v_0,\{v_{e_1},v_{e_2}\})$ and $K_{3,1}$ be the minimal element of $\mathcal H^{\free}_{\Gamma,3}(v_0,\{v_{e_1},v_{e_2}\})$.

 We have three cases to consider.
 
\begin{enumerate}

    \item We have that
    $E(v_{e_i},\{v_{e_i}\}^c)\cap E(K_{2,1},K_{2,1}^c)\neq \emptyset$, for every $i=1,2$.

    \item We have that 
    \[
    E(v_{e_1},\{v_{e_1}\}^c)\cap E(K_{2,1},K_{2,1}^c)\neq \emptyset.
    \]
    \[
    E(v_{e_2},\{v_{e_2}\}^c)\cap E(K_{2,1},K_{2,1}^c)= \emptyset
    \]
    We distinguish 3 subcases:
      \begin{enumerate}
           \item[(2.a)] There exists a $3$-hemisphere $K_3$ containing $v_{e_1}$ and $v_{e_2}$ and not containing $v_0$ such that $E(v_{e_i},\{v_{e_i}\}^c)\cap E(K_{3},K_{3}^c)\neq \emptyset$, for every $i=1,2$.

          \item[(2.b)] Every $3$-hemisphere $K_3$ containing $v_{e_1}$ and $v_{e_2}$ and not containing $v_0$ satisfies the condition $E(v_{e_2},\{v_{e_2}\}^c)\cap E(K_{3},K_{3}^c)=\emptyset$ and  there exists a $3$-hemisphere $K_3'$ such that $E(v_{e_1},\{v_{e_1}\}^c)\cap E(K_{3}',K_{3}'^c)\neq \emptyset$.
          
         \item[(2.c)] Every $3$-hemisphere $K_3$ containing $v_{e_1}$ and $v_{e_2}$ and not containing $v_0$ satisfies the condition $E(v_{e_i},\{v_{e_i}\}^c)\cap E(K_{3},K_{3}^c)=\emptyset$, for every $i=1,2$.
         
      \end{enumerate}
    Notice that there are no other subcases to consider, because if the two conditions:
 \[
    E(v_{e_2},\{v_{e_2}\}^c)\cap E(K_{3},K_{3}^c)\neq \emptyset.
    \]
    \[
    E(v_{e_1},\{v_{e_1}\}^c)\cap E(K_{3},K_{3}^c)= \emptyset
    \]
%and if $E(v_{e_2},v_{e_2}^c)\cap E(K_3,K_3^c)\ne\emptyset$, 
%an edge $e_2'$ incident to $v_{e_2}$ belonging to $E(K_3,K_3^c)$, 
hold for some $3$-hemisphere $K_3$,
then, by Lemma \ref{lem:intersect23}, we have that $K_{2,1}\cap K_3$ is a $3$-hemisphere (because $K_{2,1}$ is minimal) and it would satisfy the condition in case (2.a).
    
    \item We have 
    \[
    E(v_{e_i},\{v_{e_i}\}^c)\cap E(K_{2,1},K_{2,1}^c)=\emptyset,\; \text{for every } i=1,2.
    \]
    We distinguish 3 subcases:
      \begin{enumerate}
          \item[(3.a)] We have that $E(v_{e_i},\{v_{e_i}\}^c)\cap E(K_{3,1},K_{3,1}^c)\neq \emptyset$, for every $i=1,2$. 
          
          \item[(3.b)] We have that 
            \[
          E(v_{e_1},\{v_{e_1}\}^c)\cap E(K_{3,1},K_{3,1}^c)\neq \emptyset.
          \]
          \[
          E(v_{e_2},\{v_{e_2}\}^c)\cap E(K_{3,1},K_{3,1}^c)= \emptyset\]

          \item[(3.c)] We have that $E(v_{e_i},\{v_{e_i}\}^c)\cap E(K_{3,1},K_{3,1}^c)= \emptyset$, for every $i=1,2$. 
      \end{enumerate}
\end{enumerate}

We discuss the above cases. Case (1)  follows from Proposition \ref{prop:2-tail}.

   Case (2.a).  We assume that the orientation of $e_1$ and $e_2$ satisfies the condition $s(e_1),s(e_2)\in K_3$. Recall that $t_i=\ell([s(e_i),v_{e_i}])$. 
    
    We consider the refinement $\Gamma'$ of $\Gamma_0$ by adding two vertices over each edge. Of course, $\Gamma'$ is a refinement of $\Gamma$. We denote by $\psi$ the natural function
    \[
    \psi\col E(\Gamma')\ra E(\Gamma_0)
    \]
    taking an edge $e$ of $\Gamma'$ to the edge $f$ of $\Gamma_0$ if $e$ is obtained by subdividing $f$.

    For $i=1,2$, we already have the vertex $v_{e_i}$ over the edge $e_i$, so we will only add another vertex $v'_{e_i}$. As illustrated in Figure \ref{fig:orderings}, if $t_1< t_2$, the vertices over $e_1$ will be ordered as follows: $s(e_1), v_{e_1}, v_{e_1}', t(e_1)$, while the vertices over $e_2$ are ordered as follows: $s(e_2), v'_{e_2}, v_{e_2}, t(e_2)$. On the other hand, if $t_1> t_2$, then the orderings become $s(e_1), v_{e_1}', v_{e_1}, t(e_1)$ and $s(e_2), v_{e_2}, v'_{e_2}, t(e_2)$. 
     \begin{figure}[ht]
       \begin{tikzpicture}[scale=1.3]
          \begin{scope}
          \node[above] at (0,0) {$s(e_1)$};
          \draw[fill]  (0,0) circle (0.05cm);
          
          \node[above] at (1,0) {$v_{e_1}$};
          \draw[fill]  (1,0) circle (0.05cm);
          \node[above] at (2,0) {$v_{e_1}'$};
          \draw[fill]  (2,0) circle (0.05cm);
          \node[above] at (3,0) {$t(e_1)$};
          \draw[fill]  (3,0) circle (0.05cm);
           \draw (0,0) -- (3,0);
          \end{scope}
          \begin{scope}[shift={(0,-1)}]
          \node[above] at (0,0) {$s(e_2)$};
          \draw[fill]  (0,0) circle (0.05cm);
          
          \node[above] at (1,0) {$v_{e_2}'$};
          \draw[fill]  (1,0) circle (0.05cm);
          \node[above] at (2,0) {$v_{e_2}$};
          \draw[fill]  (2,0) circle (0.05cm);
          \node[above] at (3,0) {$t(e_2)$};
          \draw[fill]  (3,0) circle (0.05cm);
           \draw (0,0) -- (3,0);
          \end{scope}
          
          \begin{scope}[shift={(6,0)}]
          \begin{scope}
          \node[above] at (0,0) {$s(e_1)$};
          \draw[fill]  (0,0) circle (0.05cm);
          
          \node[above] at (1,0) {$v_{e_1}'$};
          \draw[fill]  (1,0) circle (0.05cm);
          \node[above] at (2,0) {$v_{e_1}$};
          \draw[fill]  (2,0) circle (0.05cm);
          \node[above] at (3,0) {$t(e_1)$};
          \draw[fill]  (3,0) circle (0.05cm);
           \draw (0,0) -- (3,0);
          \end{scope}
          \begin{scope}[shift={(0,-1)}]
          \node[above] at (0,0) {$s(e_2)$};
          \draw[fill]  (0,0) circle (0.05cm);
          
          \node[above] at (1,0) {$v_{e_2}$};
          \draw[fill]  (1,0) circle (0.05cm);
          \node[above] at (2,0) {$v_{e_2}'$};
          \draw[fill]  (2,0) circle (0.05cm);
          \node[above] at (3,0) {$t(e_2)$};
          \draw[fill]  (3,0) circle (0.05cm);
           \draw (0,0) -- (3,0);
          \end{scope}
          \end{scope}
          
       \end{tikzpicture}
       \caption{}
       \label{fig:orderings}
   \end{figure}
      
      Throughout the proof of case (2a), we will assume that $t_1<t_2$, leaving to the reader the case $t_1>t_2$. Our goal will be to find a length function $\ell'$ on $\G'$ so that  $(\G',\ell')$ is a model of $X$, and  the divisors $\P_H$ for every $H\in\mc F_{\Gamma'}(v_0,v_{e_1},v_{e_2})$ are principal on $X$.  This allows us to conclude the proof. Indeed, using Remark \ref{prop:quasiquasi} and Theorem \ref{thm:convert} we get that the divisor
\begin{equation}\label{eq:graph}
\mathcal{D}_{t_1,t_2}:=2p_0-p_{e_1,t_1}-p_{e_2,t_2}+\sum_{H\in \mc F_{\Gamma'}(v_0,v_{e_1},v_{e_2})}\mc P_H
\end{equation}
is $(p_0,\mu)$-quasistable. 
  Hence Theorem \ref{thm:tropgeoAbel} (1) tells us that the  blowup illustrated on the left hand side of Figure \ref{Fig:blowup} with $Z_1=Z_2=\wh Z$, where $\wh Z$ is the $3$-tail of $C$ induced by $K_3\cap V(\Gamma)$, gives rise to a resolution of the Abel map $\alpha^2_{\mc O_\C}$ locally at $(N_1,N_2)$. This is the blowup locally around $(N_1,N_2)$ prescribed by the global blowup in the statement of  Theorem \ref{thm:Abel2}.

     We proceed with the construction of the length function $\ell'$.  We write 
       \begin{equation}
            \label{eq:H23_free}
           \begin{aligned}
           \mathcal H_{\Gamma',2}^{\free}(v_0,\{v_{e_1},v_{e_2}\})&=\{H_{2,1},\ldots, H_{2,{m_2}}\}\\
           \mathcal H_{\Gamma',3}^{\free}(v_0,\{v_{e_1},v_{e_2}\})&=\{H_{3,1},\ldots, H_{3,m_3}\}.
           \end{aligned}
       \end{equation}

    %For an edge $f\in E(\Gamma)$ we will abuse notation and write $f\in E(H,H^c)$ for a hemisphere of $\widehat{\Gamma}'$ if $f\in E_{\Gamma}(H\cap V(\Gamma), V(\Gamma)\setminus H)$. 
    
    We define a sequence  $f_1,f_2,\ldots, f_k$ of edges of $\Gamma_0$ as illustrated in Figure \ref{fig:fi_sequence_t1<t2_odd}, with $f_1=e_1$ and $f_2\in E(H_{2,1},H_{2,1}^c)$, and where the other edges are chosen as follows.
    Assume that $t_1<t_2$. The edges of the sequence satisfy
    \[
    f_{2i+1},f_{2i+2}\in \psi(E(H_{2,{3i+1}},H^c_{2,{3i+1}})), \;\;
    f_{2i+1},f_{2i+2}\in \psi(E(H_{2,{3i+2}},H_{2,{3i+2}}^c))
    \]
    \[f_{2i}, f_{2i+1}\in \psi(E(H_{2,{3i}},H^c_{2,{3i}})).
    \]
    Notice that if $k$ is odd with $k=2k'+1$, then  
    \[
    \psi(E(H_{2,3i+1},H_{2,3i+1}^c))=\psi(E(H_{2,3i+2},H_{2,3i+2}^c))=\psi(E(H_{2,3i+3},H_{2,3i+3}^c))
    \]
    for every $i\geq k'$. If $k$ is even with $k=2k'$, then for every $i\geq k'$ we have  
    \[
    \psi(E(H_{2,3i},H_{2,3i}^c))=\psi(E(H_{2,3i+1},H_{2,3i+1}^c))=\psi(E(H_{2,3i+2},H_{2,3i+2}^c)).
    \]

    %Assume that $t_1>t_2$. The edges of the sequence satisfy
    %\[
     %f_{2i+1},f_{2i+2}\in \psi(E(H_{2,{3i+1}},H^c_{2,{3i+1}})), \]
    %\[
    % f_{2i},f_{2i+1}\in \psi(E(H_{2,{3i-1}},H^c_{2,{3i-1}})\cap \psi(E(H_{2,{3i}},H^c_{2,3i})).
     %\] 
   %  %Notice that if $k$ is odd, with $k=2k'+1$, then the condition in equation \eqref{eq:odd} holds.
   %If $k$ is even with $k=2k'$, then 
     % for every $i\geq k'$ we have \begin{equation}\label{eq:even'}
    %\psi(E(H_{2,3i-1},H^c_{2,3i-1}))=\psi(E(H_{2,3i},H^c_{2,3i})=E(H_{2,3i+1},H^c_{2,3i+1}).
    %\end{equation}

    Now we consider the $3$-hemispheres. If $\psi(E(H_{3,1},H_{3,1}^c))$ and  $\psi(E(H_{2,i},H_{2,i}^c))$ are disjoint for every $i$, let us define a length function $\ell'$ on the set of edges of $\Gamma'$ so that $(\Gamma',\ell')$ is a model for $X$. We will assume that $k=2k'$ is even (see Figure \ref{fig:fi_sequence_t1<t2_even}, also see Figure \ref{fig:fi_sequence_t1>t2_even} for the case $t_1>t_2$), leaving to the reader to work out the other case (see Figure \ref{fig:fi_sequence_t1<t2_odd}).    For every $e\in E(\G')$, we define:
       \[
\ell'(e)=
\begin{cases}
\begin{array}{ll}
  \frac{1-t_1}{2}   & \text{ if }e\in E(H_{2,{3i+1}},H_{2,3i+1}^c) \text{ or } e\in E(H_{2,3i+2},H_{2,3i+2}^c)\text{ for some } i<k'
  \\ 
  t_1   & \text{ if }\psi(e)=f_k\text{ and }e\notin E(H_{2,3k'-2},H_{2,3k'-2}^c)\cup E(H_{2,3k'-1},H_{2,3k'-1}^c)
  \\ 
  t_1 & \text{ if either } e\in E(H_{2,3i},H^c_{2,3i}) \text{ for some $i<k'$, or } t(e)=v_1\\
  %\text{ and } t(e)\in H^3_{3i+1}
  %\\
  1/3 & \text{ otherwise}.
\end{array}
\end{cases}
\]

For every edge $f\in E(\Gamma)$ we have that $\sum_{e\in \psi^{-1}(f)}\ell'(e)=1$. Indeed if $f=f_i$, then the sum of the lengths will be $t_1+\frac{1-t_1}{2}+\frac{1-t_1}{2}=1$, while if $f\notin\{f_1,\ldots, f_k\}$, then  the 3 edges in $\psi^{-1}(f)$ will have length $1/3$. So $(\Gamma',\ell')$ is a model of $X$. Notice that the divisors $\mathcal{P}_{H_{2,i}}$ and $\mathcal{P}_{H_{3,i}}$ are principal divisors by Remark \ref{rem:P_principal}. This conclude the proof in this case.
 \begin{figure}[ht]
       \begin{tikzpicture}[scale=1.2]
          \begin{scope}
          \draw[fill]  (0,0) circle (0.05cm);
          \node[above] at (1,0) {$v_{e_1}$};
          \draw[fill]  (1,0) circle (0.05cm);
          \draw[fill]  (2,0) circle (0.05cm);
          \draw[fill]  (3,0) circle (0.05cm);
          \node at (-0.3,0.0) {$f_1$};
          \draw (0,0) -- (3,0);
          \draw[blue] (1.1,0.5) to[bend left=20] (2.5,-1.7);
          \node at (2.1,-1.5) {$H_{2,1}$};
          \draw[blue] (2.1,0.5) to[bend left=20] (3.5,-1.7);
          \node at (3.1,-1.5) {$H_{2,2}$};
          \node[below] at (0.5,0) {$t_1$};
          \node[below] at (1.5,0) {$\frac{1-t_1}{2}$};
          \node[below] at (2.5,0) {$\frac{1-t_1}{2}$};
          \end{scope}
          
          \begin{scope}[shift={(2,-1)}]
          \draw[fill]  (0,0) circle (0.05cm);
          \draw[fill]  (1,0) circle (0.05cm);
          \draw[fill]  (2,0) circle (0.05cm);
          \draw[fill]  (3,0) circle (0.05cm);
          \draw (0,0) -- (3,0);
          \node at (-0.3,0.0) {$f_2$};
          \node[above] at (2.5,0) {$t_1$};
          \node[above] at (0.5,0) {$\frac{1-t_1}{2}$};
          \node[above] at (1.5,0) {$\frac{1-t_1}{2}$};
          \end{scope}
          
          \begin{scope}[shift={(4,0)}]
          \draw[fill]  (0,0) circle (0.05cm);
          \draw[fill]  (1,0) circle (0.05cm);
          \draw[fill]  (2,0) circle (0.05cm);
          \draw[fill]  (3,0) circle (0.05cm);
          \node at (-0.3,0.0) {$f_3$};
          \draw (0,0) -- (3,0);
          \draw[blue] (0.5,0.5) to[bend left=20] (0.5,-1.7);
          \node at (0.2,-1.5) {$H_{2,3}$};
          \draw[blue] (1.1,0.5) to[bend left=20] (2.5,-1.7);
          \node at (2.1,-1.5) {$H_{2,4}$};
          \draw[blue] (2.1,0.5) to[bend left=20] (3.5,-1.7);
          \node at (3.1,-1.5) {$H_{2,5}$};
          \node[below] at (0.5,0) {$t_1$};
          \node[below] at (1.5,0) {$\frac{1-t_1}{2}$};
          \node[below] at (2.5,0) {$\frac{1-t_1}{2}$};
          \end{scope}
          
          \begin{scope}[shift={(6,-1)}]
          \draw[fill]  (0,0) circle (0.05cm);
          \draw[fill]  (1,0) circle (0.05cm);
          \draw[fill]  (2,0) circle (0.05cm);
          \draw[fill]  (3,0) circle (0.05cm);
          \draw (0,0) -- (3,0);
          \node at (-0.3,0.0) {$f_4$};
          \node[above] at (2.5,0) {$t_1$};
          \node[above] at (0.5,0) {$\frac{1-t_1}{2}$};
          \node[above] at (1.5,0) {$\frac{1-t_1}{2}$};
          \end{scope}

       \end{tikzpicture}
       \caption{Attributing lengths to the edges of $\Gamma'$ for $t_1<t_2$ and $k$ even. In this case, $v_{e_2}$ is contained in $H_{2,i}$ for every $i=1,\ldots, m_2$.}
       \label{fig:fi_sequence_t1<t2_even}
   \end{figure}
   \begin{figure}[ht]
       \begin{tikzpicture}[scale=1.2]
          \begin{scope}
          \draw[fill]  (0,0) circle (0.05cm);
          \node[above] at (2,0) {$v_{e_1}$};
          \draw[fill]  (1,0) circle (0.05cm);
          \draw[fill]  (2,0) circle (0.05cm);
          \draw[fill]  (3,0) circle (0.05cm);
          \node at (-0.3,0.0) {$f_1$};
          \draw (0,0) -- (3,0);
          \draw[blue] (2.1,0.5) to[bend left=20] (2.1,-1.7);
          \node at (1.8,-1.5) {$H_{2,1}$};
          
          \node[below] at (0.5,0) {$\frac{t_1}{2}$};
          \node[below] at (1.5,0) {$\frac{t_1}{2}$};
          \node[below] at (2.7,0) {${1-t_1}$};
          \end{scope}
          
          \begin{scope}[shift={(2,-1)}]
          \draw[fill]  (0,0) circle (0.05cm);
          \draw[fill]  (1,0) circle (0.05cm);
          \draw[fill]  (2,0) circle (0.05cm);
          \draw[fill]  (3,0) circle (0.05cm);
          \draw (0,0) -- (3,0);
          \node at (-0.3,0.0) {$f_2$};
          \node[above] at (2.5,0) {$\frac{t_1}{2}$};
          \node[above] at (0.7,0) {${1-t_1}$};
          \node[above] at (1.5,0) {$\frac{t_1}{2}$};
          \end{scope}
          
          \begin{scope}[shift={(4,0)}]
          \draw[fill]  (0,0) circle (0.05cm);
          \draw[fill]  (1,0) circle (0.05cm);
          \draw[fill]  (2,0) circle (0.05cm);
          \draw[fill]  (3,0) circle (0.05cm);
          \node at (-0.3,0.0) {$f_3$};
          \draw (0,0) -- (3,0);
          \draw[blue] (0.3,0.5) to[bend left=20] (-0.9,-1.7);
          \node at (-1.2,-1.5) {$H_{2,2}$};
          \draw[blue] (1.3,0.5) to[bend left=20] (0.1,-1.7);
          \node at (-0.2,-1.5) {$H_{2,3}$};
          \draw[blue] (2.1,0.5) to[bend left=20] (2.1,-1.7);
          \node at (1.8,-1.5) {$H_{2,4}$};
          \node[below] at (0.5,0) {$\frac{t_1}{2}$};
          \node[below] at (1.5,0) {$\frac{t_1}{2}$};
          \node[below] at (2.7,0) {$1-t_1$};
          \end{scope}
          
          \begin{scope}[shift={(6,-1)}]
          \draw[fill]  (0,0) circle (0.05cm);
          \draw[fill]  (1,0) circle (0.05cm);
          \draw[fill]  (2,0) circle (0.05cm);
          \draw[fill]  (3,0) circle (0.05cm);
          \draw (0,0) -- (3,0);
          \node at (-0.3,0.0) {$f_4$};
          \node[above] at (2.5,0) {$\frac{t_1}{2}$};
          \node[above] at (0.7,0) {${1-t_1}$};
          \node[above] at (1.5,0) {$\frac{t_1}{2}$};
          \end{scope}
          
       \end{tikzpicture}
       \caption{Attributing lengths to the edges of $\Gamma'$ for $t_1>t_2$ and $k$ even. In this case, $v_{e_2}$ is contained in $H_{2,i}$ for every $i=1,\ldots, m_2$.}
       \label{fig:fi_sequence_t1>t2_even}
   \end{figure}
    
    We are left to consider the case in which  $\psi(E(H_{3,1},H_{3,1}^c))$ and  $\psi(E(H_{2,i},H_{2,i}^c))$ have a common edge for some $i$. In this case,  this edge must be $f_k$.
    
    We claim that $k$ is odd. First we prove that $H_{3,1}^c$ contains the vertices of $\Gamma'$ incident to $f_{2i}$ and the vertices over $f_{2i}$ for every $i$. Let us denote by $e_2,f, f_k$ the edges of $\psi(E(H_{3,1}, H_{3,1}^c))$.
     The intersection $H_{3,1}\cap H_{2,1}$ %is either a $3$-hemisphere or a $2$-hemisphere.\footnote{explicar porque (lemma?)} It
     cannot be a $2$-hemisphere, otherwise we would contradict the minimality of $H_{2,1}$. Indeed, the fact that 
     $E(v_{e_2},\{v_{e_2}\}^c)\cap E(H_{3,1},H_{3,1}^c)\neq \emptyset$ implies that 
     \[
     E(v_{e_2},\{v_{e_2}\}^c)\cap E(H_{3,1}\cap H_{2,1},(H_{3,1}\cap H_{2,1})^c)\neq \emptyset,
     \]
     so $H_{3,1}\cap H_{2,1}\subsetneqq H_{2,1}$.
    By Lemma \ref{lem:intersect23}, we see that   $\psi(E(H_{3,1}\cap H_{2,1}, (H_{3,1}\cap H_{2,1})^c))=\{e_2, f, f_1\}$, with $e_2$ and $f$ fully contained in $H_{2,1}$ and $f_2$ fully contained in $H_{3,1}^c$. 
     
     We now iterate the reasoning. Intersecting $H_{3,1}\cap H_{2,3}$ (see Figure \ref{fig:fi_sequence_t1<t2_odd}), we must have that $e_2, f\in \psi(E(H_{3,1}\cap H_{2,3},(H_{3,1}\cap H_{2,3})^c))$, hence, by Lemma \ref{lem:intersect23}, $H_{3,1}\cap H_{2,3}$ is a $3$-hemisphere, and $f_3$ is fully contained in $H_{3,1}$ (as neither $f_2$ nor $f_k$ is fully contained in $H_{3,1}$). Considering $H_{3,1}\cap H_{2,4}$, we have that $f_4$ must be fully contained in $H_{3,1}^c$, and iterating this process we see that $f_{2i}$ is fully contained in $H_3^c$ for every $i=1,\ldots, \lfloor \frac{k}{2}\rfloor$. So $k$ must be odd and we write $k=2k'+1$.\par
     As illustrated in Figure \ref{fig:fi_sequence_t1<t2_odd}, for every $e\in E(\G')$, we define:
     %We define the length function $\ell'$ on $\Gamma'$ as
            \[
\ell'(e)=
\begin{cases}
\begin{array}{ll}
  \frac{1-t_1}{2}   & \text{ if }e\in E(H_{2,3i+1},H_{2,3i+1}^c) \text{ or } E(H_{2,3i+2},H_{2,3i+2}^c)\text{ for some } i<k',
  \\ 
   t_1 & \text{ if } e\in E(H_{2,3i},H_{2,3i}^c) \text{ or } t(e)=v_1,\\
  %\text{ and } t(e)\in H^3_{3i+1}
  %\\
  1-t_2 &\text{ if } e\in E(H_{3,3i+1},H_{2,3i+1}^c) \text{ for any } i,\\
  t_2-t_1 &\text{ if } e\in E(H_{3,3i+2},H_{2,3i+2}^c) \text{ for any } i,\\
  t_1 &\text{ if } e\in E(H_{3,{3i}},H^c_{2,{3i}}) \text{ for any } i,\\
  1/3 & \text{ if } \psi(e)\notin \psi(E(H_{j,i},H^c_{j,i})) \text{ for any $j=2,3$ and $i$}.
\end{array}
\end{cases}
\]
     The remaining edges can be assigned lengths in a such way that $\sum_{e\in\psi^{-1}(f)}\ell'(e)=1$ for every $f\in E(\Gamma_0)$, so $(\Gamma',\ell')$ is a model of $X$. Again, by Remark \ref{rem:P_principal}, the divisors $\mathcal{P}_{H_{2,i}}$ and $\mathcal{P}_{H_{3,i}}$ are principal divisors, finishing the proof.
     
      \begin{figure}[ht]
       \begin{tikzpicture}[scale=1.2]
          \begin{scope}
          \draw[fill]  (0,0) circle (0.05cm);
          \node[above] at (1,0) {$v_{e_1}$};
          \draw[fill]  (1,0) circle (0.05cm);
          \draw[fill]  (2,0) circle (0.05cm);
          \draw[fill]  (3,0) circle (0.05cm);
          \node at (-0.3,0.0) {$f_1$};
          \draw (0,0) -- (3,0);
          \draw[blue] (1.1,0.5) to[bend left=20] (2.5,-1.7);
          \node at (2.1,-1.5) {$H_{2,1}$};
          \draw[blue] (2.1,0.5) to[bend left=20] (3.5,-1.7);
          \node at (3.1,-1.5) {$H_{2,2}$};
          \node[below] at (0.5,0) {$t_1$};
          \node[below] at (1.5,0) {$\frac{1-t_1}{2}$};
          \node[below] at (2.5,0) {$\frac{1-t_1}{2}$};
          \end{scope}
          
          \begin{scope}[shift={(2,-1)}]
          \draw[fill]  (0,0) circle (0.05cm);
          \draw[fill]  (1,0) circle (0.05cm);
          \draw[fill]  (2,0) circle (0.05cm);
          \draw[fill]  (3,0) circle (0.05cm);
          \draw (0,0) -- (3,0);
          \node at (-0.3,0.0) {$f_2$};
          \node[above] at (2.5,0) {$t_1$};
          \node[above] at (0.5,0) {$\frac{1-t_1}{2}$};
          \node[above] at (1.5,0) {$\frac{1-t_1}{2}$};
          \end{scope}
          
          \begin{scope}[shift={(4,0)}]
          \draw[fill]  (0,0) circle (0.05cm);
          \draw[fill]  (1,0) circle (0.05cm);
          \draw[fill]  (2,0) circle (0.05cm);
          \draw[fill]  (3,0) circle (0.05cm);
          \node at (-0.3,0.0) {$f_3$};
          \draw (0,0) -- (3,0);
          \draw[blue] (0.5,0.5) to[bend left=20] (0.5,-1.7);
          \node at (0.2,-1.5) {$H_{2,3}$};
          \draw[blue] (1.1,0.5) to[bend left=20] (2.5,-1.7);
          \node at (2.1,-1.5) {$H_{2,4}$};
          \draw[blue] (2.1,0.5) to[bend left=20] (3.5,-1.7);
          \node at (3.1,-1.5) {$H_{2,5}$};
          \node[below] at (0.5,0) {$t_1$};
          \node[below] at (1.5,0) {$\frac{1-t_1}{2}$};
          \node[below] at (2.5,0) {$\frac{1-t_1}{2}$};
          \end{scope}
          
          \begin{scope}[shift={(6,-1)}]
          \draw[fill]  (0,0) circle (0.05cm);
          \draw[fill]  (1,0) circle (0.05cm);
          \draw[fill]  (2,0) circle (0.05cm);
          \draw[fill]  (3,0) circle (0.05cm);
          \draw (0,0) -- (3,0);
          \node at (-0.3,0.0) {$f_4$};
          \node[above] at (2.5,0) {$t_1$};
          \node[above] at (0.5,0) {$\frac{1-t_1}{2}$};
          \node[above] at (1.5,0) {$\frac{1-t_1}{2}$};
          \end{scope}
          
          \begin{scope}[shift={(8,0)}]
          \draw[fill]  (0,0) circle (0.05cm);
          \draw[fill]  (1,0) circle (0.05cm);
          \draw[fill]  (2,0) circle (0.05cm);
          \draw[fill]  (3,0) circle (0.05cm);
          \node at (-0.3,0.0) {$f_5$};
          \draw (0,0) -- (3,0);
          \draw[blue] (0.5,0.5) to[bend left=20] (0.5,-1.7);
          \node at (0.2,-1.5) {$H_{2,6}$};
          \draw[green!40!black] (1.4,0.5) to[bend left=10] (0.1,-3.7);
          \node at (-0.2,-3.7) {$H_{3,1}$};
          \draw[green!40!black] (2.2,0.5) to[bend left=10] (0.8,-4.7);
          \node at (0.5,-4.5) {$H_{3,2}$};
          \node[below] at (0.5,0) {$t_1$};
          \node[below] at (1.7,0) {$1-t_2$};
          \node[below] at (2.75,0) {$t_2-t_1$};
          \end{scope}

          \begin{scope}[shift={(8.5,-2)}]
          \draw[fill]  (0,0) circle (0.05cm);
          \draw[fill]  (1,0) circle (0.05cm);
          \draw[fill]  (2,0) circle (0.05cm);
          \draw[fill]  (3,0) circle (0.05cm);
          \draw (0,0) -- (3,0);
          \node[below] at (2.6,0) {$t_1$};
          \node[below] at (0.65,0) {$1-t_2$};
          \node[below] at (1.75,0) {$t_2-t_1$};
          \end{scope}
          
          \begin{scope}[shift={(6.2,-3)}]
          \draw[fill]  (0,0) circle (0.05cm);
          \draw[fill]  (1,0) circle (0.05cm);
          \draw[fill]  (2,0) circle (0.05cm);
          \draw[fill]  (3,0) circle (0.05cm);
          \draw (0,0) -- (3,0);
          \node[above] at (2,0) {$v_{e_2}$};
          \node[below] at (1.5,0) {$t_1$};
          \node[below] at (2.65,0) {$1-t_2$};
          \node[below] at (0.5,0) {$t_2-t_1$};
          \end{scope}
          
          \begin{scope}[shift={(8.9,-4)}]
          \draw[fill]  (0,0) circle (0.05cm);
          \draw[fill]  (1,0) circle (0.05cm);
          \draw[fill]  (2,0) circle (0.05cm);
          \draw[fill]  (3,0) circle (0.05cm);
          \draw (0,0) -- (3,0);
          \node[below] at (1.5,0) {$t_1$};
          \node[below] at (2.5,0) {$1-t_2$};
          \node[below] at (0.65,0) {$t_2-t_1$};
          \end{scope}
       \end{tikzpicture}
       \caption{Attributing lengths to the edges of $\Gamma'$, for $t_1<t_2$ and $k$ odd. In this case, $v_{e_2}$ is contained in $H_{2,i}$ for every $i=1,\ldots, m_2$. }
       \label{fig:fi_sequence_t1<t2_odd}
   \end{figure}

   \begin{figure}[ht]
       \begin{tikzpicture}[scale=1.2]
          \begin{scope}
          \draw[fill]  (0,0) circle (0.05cm);
          \node[above] at (2,0) {$v_{e_1}$};
          \draw[fill]  (1,0) circle (0.05cm);
          \draw[fill]  (2,0) circle (0.05cm);
          \draw[fill]  (3,0) circle (0.05cm);
          \node at (-0.3,0.0) {$f_1$};
          \draw (0,0) -- (3,0);
          \draw[blue] (2.1,0.5) to[bend left=20] (2.1,-1.7);
          \node at (1.8,-1.5) {$H_{2,1}$};
          
          \node[below] at (0.5,0) {$\frac{t_1}{2}$};
          \node[below] at (1.5,0) {$\frac{t_1}{2}$};
          \node[below] at (2.7,0) {${1-t_1}$};
          \end{scope}
          
          \begin{scope}[shift={(2,-1)}]
          \draw[fill]  (0,0) circle (0.05cm);
          \draw[fill]  (1,0) circle (0.05cm);
          \draw[fill]  (2,0) circle (0.05cm);
          \draw[fill]  (3,0) circle (0.05cm);
          \draw (0,0) -- (3,0);
          \node at (-0.3,0.0) {$f_2$};
          \node[above] at (2.5,0) {$\frac{t_1}{2}$};
          \node[above] at (0.7,0) {${1-t_1}$};
          \node[above] at (1.5,0) {$\frac{t_1}{2}$};
          \end{scope}
          
          \begin{scope}[shift={(4,0)}]
          \draw[fill]  (0,0) circle (0.05cm);
          \draw[fill]  (1,0) circle (0.05cm);
          \draw[fill]  (2,0) circle (0.05cm);
          \draw[fill]  (3,0) circle (0.05cm);
          \node at (-0.3,0.0) {$f_3$};
          \draw (0,0) -- (3,0);
          \draw[blue] (0.3,0.5) to[bend left=20] (-0.9,-1.7);
          \node at (-1.2,-1.5) {$H_{2,2}$};
          \draw[blue] (1.3,0.5) to[bend left=20] (0.1,-1.7);
          \node at (-0.2,-1.5) {$H_{2,3}$};
          \draw[blue] (2.1,0.5) to[bend left=20] (2.1,-1.7);
          \node at (1.8,-1.5) {$H_{2,4}$};
          \node[below] at (0.5,0) {$\frac{t_1}{2}$};
          \node[below] at (1.5,0) {$\frac{t_1}{2}$};
          \node[below] at (2.7,0) {$1-t_1$};
          \end{scope}
          
          \begin{scope}[shift={(6,-1)}]
          \draw[fill]  (0,0) circle (0.05cm);
          \draw[fill]  (1,0) circle (0.05cm);
          \draw[fill]  (2,0) circle (0.05cm);
          \draw[fill]  (3,0) circle (0.05cm);
          \draw (0,0) -- (3,0);
          \node at (-0.3,0.0) {$f_4$};
          \node[above] at (2.5,0) {$\frac{t_1}{2}$};
          \node[above] at (0.7,0) {${1-t_1}$};
          \node[above] at (1.5,0) {$\frac{t_1}{2}$};
          \end{scope}
          
          \begin{scope}[shift={(8,0)}]
          \draw[fill]  (0,0) circle (0.05cm);
          \draw[fill]  (1,0) circle (0.05cm);
          \draw[fill]  (2,0) circle (0.05cm);
          \draw[fill]  (3,0) circle (0.05cm);
          \node at (-0.3,0.0) {$f_5$};
          \draw (0,0) -- (3,0);
          \draw[blue] (0.3,0.5) to[bend left=20] (-0.9,-1.7);
          \node at (-1.2,-1.5) {$H_{2,5}$};
          \draw[blue] (1.3,0.5) to[bend left=20] (0.1,-1.7);
          \node at (-0.2,-1.5) {$H_{2,6}$};
          
          \draw[green!40!black] (2.4,0.5) to[bend left=10] (-0.3,-3.7);
          \node at (-0.4,-3.8) {$H_{3,1}$};
          \node[below] at (0.5,0) {$\frac{t_1}{2}$};
          \node[below] at (1.5,0) {$\frac{t_1}{2}$};
          \node[below] at (2.7,0) {$1-t_1$};
          \end{scope}

          \begin{scope}[shift={(8.7,-2)}]
          \draw[fill]  (0,0) circle (0.05cm);
          \draw[fill]  (1,0) circle (0.05cm);
          \draw[fill]  (2,0) circle (0.05cm);
          \draw[fill]  (3,0) circle (0.05cm);
          \draw (0,0) -- (3,0);
          \draw[green!40!black] (1.5,0.5) to[bend left=10] (-0.8,-2.7);
          \node at (-1.1,-2.7) {$H_{3,2}$};
          \node[below] at (2.6,0) {$t_2$};
          \node[below] at (0.7,0) {$1-t_1$};
          \node[below] at (1.7,0) {$t_1-t_2$};
          
          \end{scope}
          
          \begin{scope}[shift={(7,-3)}]
          \draw[fill]  (0,0) circle (0.05cm);
          \draw[fill]  (1,0) circle (0.05cm);
          \draw[fill]  (2,0) circle (0.05cm);
          \draw[fill]  (3,0) circle (0.05cm);
          \draw (0,0) -- (3,0);
          
          \node[above] at (1,0) {$v_{e_2}$};
          \node[below] at (1.7,0) {$1-t_1$};
          \node[below] at (2.7,0) {$t_1-t_2$};
          \node[below] at (0.5,0) {$t_2$};
          \end{scope}
          
          \begin{scope}[shift={(8.4,-4)}]
          \draw[fill]  (0,0) circle (0.05cm);
          \draw[fill]  (1,0) circle (0.05cm);
          \draw[fill]  (2,0) circle (0.05cm);
          \draw[fill]  (3,0) circle (0.05cm);
          \draw (0,0) -- (3,0);
          \node[below] at (1.5,0) {$t_2$};
          \node[below] at (2.5,0) {$1-t_1$};
          \node[below] at (0.65,0) {$t_1-t_2$};
          \end{scope}
          
       \end{tikzpicture}
       \caption{Attributing lengths to the edges of $\Gamma'$ for $t_1>t_2$ and $k$ odd. In this case, $v_{e_2}$ is contained in $H_{2,i}$ for every $i=1,\ldots, m_2$.}
       \label{fig:fi_sequence_t1>t2_odd}
   \end{figure}

Case (2.b). Consider the refinement $\Gamma'$ of $\Gamma_0$ by adding one vertex over each edge. Notice that $\G'$ is a refinement of $\G$.    Let $H_{2,i}$ and $H_{3,j}$ be defined as in Equation \eqref{eq:H23_free}. Let $k$ be the integer such that $|\psi(E(H_{2,i},H_{2,i}^c))\cap\psi(E(H_{2,i+1},H_{2,i+1}^c))|=1$ for every $i\leq k-1$, and 
\[
\psi(E(H_{2,k+2i+1},H_{2,k+2i+1}^c))=\psi(E(H_{2,k+2i+2},H_{2,k+2i+2}^c)), \text{ for } i\ge 0.
\]
 If $k$ is even, we define the length $\ell'$ on $\Gamma'$ as follows:
\[
\ell'(e)=
\begin{cases}
\begin{array}{ll}
  1-t_1   & \text{ if }e\in E(H_{2,2i+1},H_{2,2i+1}) \text{ with } i=0,\dots,\lfloor\frac{m_2-1}{2}\rfloor 
  \\ 
   t_1 & \text{ if } e\in E(H_{2,2i},H_{2,2i}^c)  \text{ with } i=1,\dots,\lfloor\frac{m_2}{2}\rfloor  \text{ or } t(e)=v_1\\
  %\text{ and } t(e)\in H^3_{3i+1}
  %\\
  1-t_1 &\text{ if } e\in E(H_{3,2i},H_{3,2i}^c) \text{ with } i=1,\dots,\lfloor\frac{m_3}{2}\rfloor\\
  t_1 &\text{ if } e\in E(H_{3,2i+1},H_{3,2i+1}^c) \text{ with } i=0,\dots,\lfloor\frac{m_3-1}{2}\rfloor \\
  1/2 & \text{ if } \psi(e)\notin \psi(E(H_{r,i},H_{r,i}^c) \text{ for any $r=2,3$ and $i$}.
\end{array}
\end{cases}
\]
The remaining edges can be assigned lengths in a such way that $\sum_{e\in\psi^{-1}(f)} \ell'(e)=1$ for every $f\in E(\Gamma_0)$, so $(\Gamma',\ell')$ is a model of $X$. 
   
When $k$ is odd the situations is similar (see Figure \ref{fig:2b}): the unique difference is that we define $\ell'(e)=1-t_1$ for $e\in E(H_{3,2i+1},H_{3,2i+1}^c)$ and $\ell'(e)=t_1$ for $e\in E(H_{3,2i},H_{3,2i}^c)$.
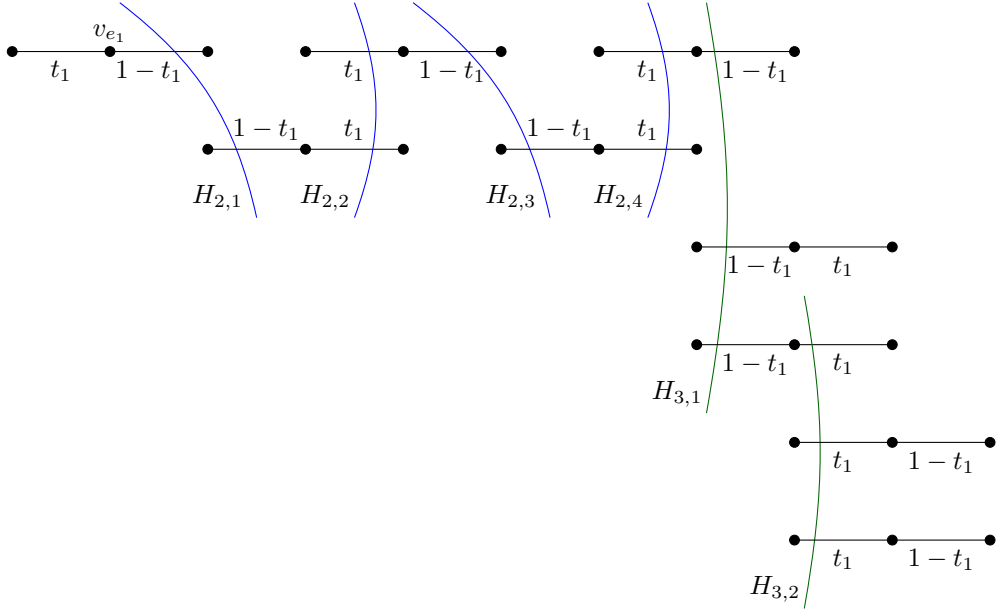
\begin{figure}[ht]
       \begin{tikzpicture}[scale=1.3]
          \begin{scope}
          \draw[fill]  (0,0) circle (0.05cm);
          \node[above] at (1,0) {$v_{e_1}$};
          \draw[fill]  (1,0) circle (0.05cm);
          \draw[fill]  (2,0) circle (0.05cm);
          \draw (0,0) -- (2,0);
          \draw[blue] (1.1,0.5) to[bend left=20] (2.5,-1.7);
          \node at (2.1,-1.5) {$H_{2,1}$};
          \node[below] at (0.5,0) {$t_1$};
          \node[below] at (1.4,0) {$1-t_1$};
          \end{scope}
          
          \begin{scope}[shift={(2,-1)}]
          \draw[fill]  (0,0) circle (0.05cm);
          \draw[fill]  (1,0) circle (0.05cm);
          \draw[fill]  (2,0) circle (0.05cm);
          \draw (0,0) -- (2,0);
          \node[above] at (1.5,0) {$t_1$};
          \node[above] at (0.6,0) {$1-t_1$};
          \end{scope}
          
          \begin{scope}[shift={(3,0)}]
          \draw[fill]  (0,0) circle (0.05cm);
          \draw[fill]  (1,0) circle (0.05cm);
          \draw[fill]  (2,0) circle (0.05cm);
          \draw (0,0) -- (2,0);
          \draw[blue] (0.5,0.5) to[bend left=20] (0.5,-1.7);
          \node at (0.2,-1.5) {$H_{2,2}$};
          \draw[blue] (1.1,0.5) to[bend left=20] (2.5,-1.7);
          \node at (2.1,-1.5) {$H_{2,3}$};
          \node[below] at (0.5,0) {$t_1$};
          \node[below] at (1.5,0) {$1-t_1$};
          \end{scope}
          
          \begin{scope}[shift={(5,-1)}]
          \draw[fill]  (0,0) circle (0.05cm);
          \draw[fill]  (1,0) circle (0.05cm);
          \draw[fill]  (2,0) circle (0.05cm);
          \draw (0,0) -- (2,0);
          \node[above] at (1.5,0) {$t_1$};
          \node[above] at (0.6,0) {$1-t_1$};
          \end{scope}
          
          \begin{scope}[shift={(6,0)}]
          \draw[fill]  (0,0) circle (0.05cm);
          \draw[fill]  (1,0) circle (0.05cm);
          \draw[fill]  (2,0) circle (0.05cm);
          \draw (0,0) -- (2,0);
          \draw[blue] (0.5,0.5) to[bend left=20] (0.5,-1.7);
          \node at (0.2,-1.5) {$H_{2,4}$};
          \draw[green!40!black] (1.1,0.5) to[bend left=10] (1.1,-3.7);
          \node at (0.8,-3.5) {$H_{3,1}$};
          \node[below] at (0.5,0) {$t_1$};
          \node[below] at (1.6,0) {$1-t_1$};
          \end{scope}

          \begin{scope}[shift={(7,-2)}]
          \draw[fill]  (0,0) circle (0.05cm);
          \draw[fill]  (1,0) circle (0.05cm);
          \draw[fill]  (2,0) circle (0.05cm);
          \draw (0,0) -- (2,0);
          \node[below] at (1.5,0) {$t_1$};
          \node[below] at (0.65,0) {$1-t_1$};
          \end{scope}
          
          \begin{scope}[shift={(7,-3)}]
          \draw[fill]  (0,0) circle (0.05cm);
          \draw[fill]  (1,0) circle (0.05cm);
          \draw[fill]  (2,0) circle (0.05cm);
          \draw (0,0) -- (2,0);
          \node[below] at (1.5,0) {$t_1$};
          \node[below] at (0.6,0) {$1-t_1$};
          \draw[green!40!black] (1.1,0.5) to[bend left=10] (1.1,-2.7);
          \node at (0.8,-2.5) {$H_{3,2}$};
          \end{scope}
          
          \begin{scope}[shift={(8,-4)}]
          \draw[fill]  (0,0) circle (0.05cm);
          \draw[fill]  (1,0) circle (0.05cm);
          \draw[fill]  (2,0) circle (0.05cm);
          \draw (0,0) -- (2,0);
          \node[below] at (0.5,0) {$t_1$};
          \node[below] at (1.5,0) {$1-t_1$};
          \end{scope}
          
          \begin{scope}[shift={(8,-5)}]
          \draw[fill]  (0,0) circle (0.05cm);
          \draw[fill]  (1,0) circle (0.05cm);
          \draw[fill]  (2,0) circle (0.05cm);
          \draw (0,0) -- (2,0);
          \node[below] at (0.5,0) {$t_1$};
          \node[below] at (1.5,0) {$1-t_1$};
          \end{scope}
       \end{tikzpicture}
       \caption{Attributing lengths to the edges of $\Gamma'$ for $k$ odd.}
       \label{fig:2b}
   \end{figure}
   
   As in Case (2.a), we have that the combinatorial type of the divisor $\mathcal{D}_{t_1,t_2}$ defined in Equation \eqref{eq:graph} does not depend on $0<t_1,t_2<1$. Hence Theorem \ref{thm:tropgeoAbel}  (3) ensures that the Abel map $\alpha^2_{\mc O_{\C}}$ is already defined at $(N_1,N_2)$, as given by the global blowup in the statement of Theorem \ref{thm:Abel2}.\\

Case (2.c). This case is the same as Case (2.b) except that $\psi(E(H_{2,i},H^c_{2,i})\cap \psi(E(H_{3,j},H^c_{3,j})=\emptyset$ for every $i=1,\ldots,m_2$ and $j=1,\ldots, m_3$. So we can freely assign lengths to the edges in $E(H_{3,j},H^c_{3,j})$. The conclusion is the same as in Case (2.b).

\bigskip
Case (3.a). This case follows the same steps in Case (2.a): the difference is that $k=0$ and the sequence of edges $f_1,\ldots, f_k$ is empty. The conclusion is the same as in Case (2.a).

\medskip

Case (3.b). This case follows the same steps in Case (2.b): the difference is that $k=0$. The conclusion is the same as in Case (2.b).

\medskip

Case (3.c). In this case, we do not have to further refine $\Gamma_0$ as  $E(H,H^c)$ does not contain any edge incident to $v_{e_1}$ or $v_{e_2}$, for every $H$ in $\mc F_{\Gamma}(v_0,v_{e_1},v_{e_2})$. So $P_H$ is principal on $X$ for every $H\in \mc F_{\Gamma}(v_0,v_{e_1},v_{e_2})$.
   %Assume that $\{N_1,N_2\}\cap Z\cap Z^c=\emptyset$, for every $\delta$-tail $Z$ of $C$, with $\delta\in\{1,2,3\}$. Equivalently, %$\{e_1,e_2\}\cap E(H,H^c)=\emptyset$ for every $\delta$-hemisphere $H$ of $\Gamma$ with $\delta\in\{1,2,3\}$. 
 %Note that every $\delta$-hemisphere $H$, for $\delta\in\{1,2,3\}$ of $\widehat{\Gamma}$ satisfies that $E(H,H^c)$ does not contain any edge over $e_1$ or $e_2$. In particular, $\widehat{\ell}(e)=1$ for every $e\in E(H,H^c)$. Also recall that $\F_{\widehat{H}}(v_0,v_{e_1},v_{e_2})$ only contains $\delta$-hemispheres for $\delta\in \{1,2,3\}$, which implies that $\P_{H}$ is a principal divisor on $X$.\par
As in the previous cases, using Remark \ref{prop:quasiquasi} and Theorem \ref{thm:convert}, the divisor
\[
\mathcal{D}_{t_1,t_2}:=2p_0-p_{e_1,t_1}-p_{e_2,t_2}+\sum_{H\in \mc F_{\Gamma}(v_0,v_{e_1},v_{e_2})}\mc P_H
\]
on the tropical curve $X$  is $(p_0,\mu)$-quasistable and equivalent to $2p_0-p_{e_1,t_1}-p_{e_2,t_2}$. Since the combinatorial type of $\mathcal{D}_{t_1,t_2}$ is independent of $t_1$ and $t_2$, it follows from Theorem \ref{thm:tropgeoAbel} (3)  that the Abel map $\alpha^2_{\mc O_{\C}}$ is already defined at $(N_1,N_2)$, as prescribed by the global blowup in the statement of Theorem \ref{thm:Abel2}. 
\end{proof}

Given a regular smoothing $f\col\mathcal C\ra B$ of a curve, consider the blowup $\wt \C\ra \C$ giving rise to a resolution of the degree-2 Abel map $\alpha^2_{\mathcal O_{\mathcal C}}$, as in Theorem \ref{thm:Abel2}. Since the locus we are blowing up is invariant under the natural action of $S_2$ on $\mathcal{C}^2$, we can take the quotient
\[
\Sym^2(\wt{\mathcal C})=\wt \C/S_2.
\]
We thus obtain a map:
\begin{equation}\label{eq:beta2}
\beta_2\col \Sym^2(\wt{\mathcal C})\ra \overline{\mc J}_\mu(\mathcal C)
\end{equation}
resolving the rational ``symmetrized" Abel map $\Sym(\C^2)\dashrightarrow \overline{\mc J}_\mu(\mathcal C)$.

\begin{Def}\label{def:pseudo}
Let $C$ be a curve. We say that $C$ is \emph{pseudo-hyperelliptic} if it has a simple torsion-free rank-1 sheaf $I$ of degree 2 with non-negative degree over every component  such that $h^0(C,I)\ge 2$.
\end{Def}

Recall that a curve is weakly-hyperelliptic if it has a degree-$2$ balanced invertible sheaf (see \cite{Capohyper} for more details). If a stable curve is hyperelliptic, then it is weakly-hyperelliptic.

\begin{Thm}\label{thm:hyperelliptic}
Let $C$ be a curve with no separating nodes. The following properties hold.
\begin{enumerate}
    \item $C$ is pseudo-hyperelliptic
%there exists a line bundle $L$ on $C$ with $h^0(L,C)\geq 2$ and $L$ having nonnegative multidegree,
if and only if, for some (every) regular smoothing $\C\to B$ of $C$, the map $\beta_2\col \Sym^2(\C)\to \overline{\J}_{\mu}(\C)$ is not injective.
\item if $C$ is stable and weakly-hyperelliptic, then $C$ is pseudo-hyperelliptic.
\item if $C$ is stable and has a simple torsion-free rank-1 sheaf $I$ of degree 2 with non-negative degree over every component  such that $h^0(C,I)\ge 2$, then $I$ is invertible.
\end{enumerate}
\end{Thm}
\begin{proof}
If $C$ has a rational component $E$ such that $|E\cap E^c|\leq 2$ then it is easy to see that $C$ is pseudo-hyperelliptic and weakly-hyperellpitic, and that $\beta_2$ is not injective. So, we will assume that $C$ is stable.

 Assume that $C$ is stable and pseudo-hyperelliptic. Let $I$ be a torsion free rank-1 sheaf satisfying the condition in Definition \ref{def:pseudo}. Let $\mathbb P(I):=\Proj(\Sym(I))\to C$ be the semistable modification of $C$ where we add a rational curve over the nodes of $C$ where $I$ is not locally free. We consider the invertible sheaf $L:=\mathcal O_{\mathbb P(I)}(1)$, so that we have $I=f_*(L)$ and $L$ has degree $1$ on the exceptional components (see \cite[Section 5]{EP}). Then, $L$ has non-negative degree on every component of $\mathbb{P}(I)$ and $h^0(\mathbb{P}(I),L)\geq 2$.

We will apply \cite[Theorem 5.9]{Capohyper} to $\mathbb{P}(I)$ and $L$.  We have two cases. 
In the first case, there is a component $C_0$ of $\mathbb{P}(I)$ satisfying the following property. Let $Z_1,\dots Z_n$ be the connected components of $C_0^c$. Then
\begin{equation}
\label{eq:weakly_hyper}
    h^0(L|_{C_0})\geq 2,\;\;\;  L|_{C_0^c}=\mathcal{O}_{C_0^c}, \;\;\;L|_{C_0}=\mathcal{O}_{C_0}(C_0\cap Z_i),\;\;\; |C_0\cap Z_i|=2.
\end{equation}
This means that the component $C_0$ is not exceptional, because $L$ has degree $2$ on $C_0$. Moreover, $L$ has degree $0$ on every other component, which implies that $I$ is an invertible sheaf, hence $L=I$ and $\mathbb{P}(I)=C$. 
We can consider smooth points $q_1,q_2,q_1',q_2'$ of $C$ lying over $C_0$ such that 
\[
L|_{C_0}\cong\mathcal{O}_{C_0}(q_1+q_2)\cong\mathcal{O}_{C_0}(q_1'+q_2').
\]
By Condition \eqref{eq:weakly_hyper}, we have that $L\cong\mathcal{O}_C(q_1+q_2)\cong\mathcal{O}_C(q'_1+q'_2)$, hence $\mathcal{O}_C(2p_0-q_1-q_2)\cong\mathcal{O}_C(2p_0-q_1'-q_2')$ which means that $\beta_2(q_1+q_2)=\beta_2(q_1'+q_2')$, where $\beta_2$ is the map in Equation \ref{eq:beta2} for some (every) regular smoothing of $C$.

In the second case, there are two components $C_1$ and $C_2$ of $\mathbb{P}(I)$ such that $(C_1,C_2)$ is a special $\mathcal B$-pair (in the sense of \cite[Definition 5.8]{Capohyper}).  By \cite[Theorem 5.9]{Capohyper}, we have
\[
\deg L_{C_1}=\deg L|_{C_2}=1,\;\;\; L|_{(C_1\cup C_2)^c}\cong \mathcal O_{(C_1\cup C_2)^c}.
\]
Notice that, if one between $C_1$ and $C_2$ is exceptional, then the other must be exceptional as well, and in particular this implies that $I$ is not simple, which is a contradiction. We deduce that $I$ is an invertible sheaf, hence $L=I$ and $\mathbb{P}(I)=C$.
We can repeat the argument used in the first case, now taking $q_1,q_1'\in C_1$ and $q_2,q_2'\in C_2$. We leave the details  to the reader. Notice that we proved (3) and the ``only if" part of (1).

Now assume that there is a regular smoothing $\mathcal C\ra B$ of $C$ such that $\beta_2$ is not injective. We have different cases to consider.

In the first case, there are smooth points $q_1,q_2,q'_1,q'_2$ of $C$ such that $\beta_2(q_1+q_2)=\beta_2(q_1'+q_2')$. This means that there exists an invertible sheaf $T$ on $C$ of type $T=\mathcal O_{\C}(\sum a_i C_i)|_C$, where $a_i\in \mathbb Z$ and $C_i$ are the components of $C$, such that 
\[
\mathcal{O}_C(2p_0-q_1-q_2)\cong \mathcal{O}_C(2p_0-q_1'-q_2')\otimes T.
\]
We deduce that
\[
\mathcal{O}_C(q_1'+q_2'-q_1-q_2)\cong T.
\]
Let $\Gamma$ be the dual graph of $C$. Notice that $\Gamma$ has no separating edge. 
If $v_1,v_2,v_1',v_2'$ are the vertices of $\Gamma$ corresponding to the components containing the points $q_1,q_2, q_1',q_2'$, we have that $v_1'+v_2'-v_1-v_2$ is a principal divisor on $\Gamma$. Let $f\col V(\Gamma)\to \mathbb{Z}$ be the rational function on $\Gamma$ such that $\div(f)=v_1'+v_2'-v_1-v_2$ (notice that $v'_i\neq v_j$ for every $i,j=1,2$ because $\Gamma$ has no separating edge). We denote by $Z$ the subcurve of $C$ corresponding to the vertices of $\Gamma$ where $f$ attains its minimum. In particular, $q_1,q_2\in Z$, $q_1',q_2'\in Z^c$ and $|Z\cap Z^c|=2$. Moreover we have $T|_Z=\mathcal{O}_Z(-Z\cap Z^c)$, which implies that 
\[
\mathcal{O}_Z(Z\cap Z^c)\otimes\mathcal O_C(-q_1-q_2)|_Z\cong (\mathcal{O}_C(q_1'+q_2'-q_1-q_2)\otimes T^{-1})|_{Z}\cong \mathcal{O}_C|_Z=\mathcal{O}_Z.
\]
Define $L:=\mathcal{O}_C(q_1+q_2)$. We see that $L$ satisfies $h^0(L,C)\geq 2$ (indeed $L$ has the trivial section that vanishes only over $q_1,q_2$ and a section that vanishes on the whole $Z^c$). Thus $C$ is pseudo-hyperellitic in the sense of Definition \ref{def:pseudo}.

 In the second case, we have nodes $n,n'$ and smooth points  $q,q'$ of $C$ such that $\beta_2(n+q)=\beta_2(n'+q')$. Let $\wt C$ be the semistable modification of $C$ obtained by adding an exceptional component over each node of $C$. Then, there exists a twister $T$ on $\wt C$ such that 
 \[
 \mathcal{O}_{\wt{C}}(2p_0-\wt{n}-q)\cong \mathcal{O}_{\wt{C}}(2p_0-\wt{n}'-q')\otimes T
 \]
 where $\wt n$ and $\wt n'$ are any smooth points of $\widetilde C$ lying over the exceptional component over $n$ and $n'$.
 Arguing as before, we see that $L:=\mathcal{O}_{\wt{C}}(\wt{n}+q)$ satisfies $h^0(L,\wt{C})\geq 2$, and hence $h^0(f_*(L),C)\geq2$. Thus $C$ is pseudo-hyperelliptic in the sense of Definition \ref{def:pseudo}.
 
 In the third case, we have a node $n$ and smooth points $q,q'_1,q'_2$ such that $\beta_2(n+q)=\beta_2(q'_1+q'_2)$. This case is not possible, since the sheaf represented by  $\beta_2(q'_1+q'_2)$ is invertible, while the one represented by $\beta_2(n+q)$ is not.

%TO DO:
%\begin{enumerate}
%    \item $(n_1,n_2)\sim (q_1',q_2')$
 %   \item $(n_1,n_2)\sim (n',q')$
  %  \item $(n_1,n_2)\sim (n_1',n_2')$
%\end{enumerate}

The remaining cases are the following ones:
\begin{itemize}
    \item $\beta(n_1+n_2)=\beta(q'_1+q'_2)$;
    \item $\beta(n_1+n_2)=\beta(n'+q')$;
    \item $\beta (n_1+n_2)=\beta (n_1'+n_2')$.
\end{itemize}
where $n_1$, $n_2$, $n'$, $n'_1$, $n'_2$ are nodes of $C$, and $q$, $q'_1$, $q'_2$ and smooth points.
All these cases are done in an similar manner as the second case: first, we change $C$ by a suitable semistable modification $f\col \widetilde{C}\to C$ and find a line bundle $L$ such that $h^0(L,\widetilde{C})\geq 2$. This implies that $h^0(f_*(L),C)\geq 2$ which proves that $C$ is pseudo-hyperelliptic in the sense of Definition \ref{def:pseudo}.

Finally, item (2) of the statement readily follows  by \cite[Theorem 5.9]{Capohyper}.
\end{proof}

\bibliographystyle{abbrv}
\bibliography{bibliography}

\end{document}